\documentclass{au}
 \presentedby{\dots}
 \received{\dots}{\dots}

\usepackage{pgf,pgfarrows,pgfnodes,pgfautomata,pgfheaps,pgfshade}
\usepackage{enumerate}
\usepackage[utf8]{inputenc}
\usepackage{dsfont,bbm,mathrsfs}
\usepackage{amsmath, amsthm, amsbsy, amssymb}
\usepackage[all]{xy}
\usepackage{subfigure}

\renewcommand{\leq}{\leqslant}
\renewcommand{\geq}{\geqslant}

\DeclareMathOperator{\Z}{\mathbb{Z}}
\DeclareMathOperator{\Q}{\mathbb{Q}}
\DeclareMathOperator{\R}{\mathbb{R}}

\DeclareMathOperator{\VV}{\mathbb{V}} 
\DeclareMathOperator{\WW}{\mathbb{W}} 
\DeclareMathOperator{\MM}{\mathbb{M}} 
\DeclareMathOperator{\KK}{\mathbb{K}} 
\DeclareMathOperator{\KKf}{\mathbb{K}_{\rm f}} 
\DeclareMathOperator{\KS}{{\sf KS}} 
\DeclareMathOperator{\F}{\mathscr{F}} 

\DeclareMathOperator{\SC}{\mathscr{S}} 

\DeclareMathOperator{\m}{\frac{1}{2}} 
\DeclareMathOperator{\Sol}{\mathsf{Sol}} 
\DeclareMathOperator{\Ker}{\rm ker} 
\DeclareMathOperator{\conv}{\mathrm{conv}} 
\DeclareMathOperator{\den}{\mathrm{den}} 
\DeclareMathOperator{\ver}{\mathrm{ver}} 
\DeclareMathOperator{\rel}{\mathrm{relint}} 

\DeclareMathOperator{\N}{\mathscr{N}} 
\DeclareMathOperator{\McN}{\mathcal M}

\newtheorem{theorem}{Theorem}[section]

\newtheorem*{theoremI}{Theorem I}
\newtheorem*{theoremII}{Theorem II}

\newtheorem{lemma}[theorem]{Lemma}

\newtheorem*{claimnonum}{Claim}

\theoremstyle{definition}
\newtheorem{question}{Question}

\newtheorem{definition}[theorem]{Definition}
\newtheorem{remark}[theorem]{Remark}
\newtheorem*{remarknonum}{Remark}
\newtheorem*{remarks}{Remarks}

\begin{document}
\title[MV-algebras free over Kleene algebras]{MV-algebras   freely generated by finite  Kleene algebras}
\author[S. Aguzzoli] 
{Stefano Aguzzoli}
\email{aguzzoli@dsi.unimi.it}
\address{Dipartimento di Scienze dell'Informazione, Universit\`a degli
Studi di Milano. Via Comelico 39--41, 20135 Milano, Italy}

\author[L. M. Cabrer] 
{Leonardo M. Cabrer}
\email{lmcabrer@yahoo.com.ar}
\address{Mathematisches Institut, Universit\"at Bern.
Sidlerstrasse 5, CH-3012 Bern, Switzerland}

\author[V. Marra] 
{Vincenzo Marra}
\email{vincenzo.marra@unimi.it}
\address{Dipartimento di Matematica ``Federigo Enriques'', Universit\`a degli
Studi di Milano. Via Saldini 50, 20133 Milano, Italy}



\begin{abstract}If $\VV$ and $\WW$ are varieties of algebras such that any $\VV$-algebra $A$ has a reduct $U(A)$ in $\WW$, there is a forgetful functor $U \colon \VV \to \WW$ that acts by $A \mapsto U(A)$ on objects, and  identically on homomorphisms. This functor $U$ always has a left adjoint $F\colon \WW \to \VV$ by general considerations.  One calls $F(B)$ the \emph{$\VV$-algebra freely generated
by the $\WW$-algebra $B$}. Two problems  arise naturally in this broad setting. The {\it description problem} is to  describe the structure of the $\VV$-algebra $F(B)$ as explicitly as possible in terms of the structure of the $\WW$-algebra $B$.
The {\it recognition problem} is to find  conditions on the structure of a given $\VV$-algebra $A$ that are necessary and sufficient for the existence of a   $\WW$-algebra $B$ such that $F(B)\cong A$. Building on and extending previous work on MV-algebras freely generated by finite distributive lattices, in this paper we provide solutions to the description and recognition problems in case $\VV$ is the variety of MV-algebras, $\WW$ is the variety of Kleene algebras, and $B$ is finitely
generated---equivalently, finite.
The proofs rely heavily on the Davey-Werner natural duality for Kleene algebras,  on the  representation of finitely presented MV-algebras by  compact rational polyhedra, and on the theory of bases of MV-algebras.
\end{abstract}

\maketitle

\section{Introduction.}\label{s:introduction}
Consider a variety of algebras $\VV$, and write $\F_{\kappa}^{\VV}$ for the  algebra in $\mathbb{V}$ freely generated by  a set of cardinality $\kappa$. Suppose further that $\WW$ is a variety such that any $\VV$-algebra $A$ has a reduct $U(A)$ in $\WW$. Then there exists a forgetful functor $U \colon \VV \to \WW$ that acts by $A \mapsto U(A)$ on objects, and  identically on homomorphisms. This functor $U$ always has a left adjoint, as follows. Let $B$ be any $\WW$-algebra, and say its cardinality is $\kappa=|B|$. Then $B$ is isomorphic to a quotient $\F_{\kappa}^{\WW}/\Theta$, for some congruence $\Theta\subseteq \F_{\kappa}^{\WW}\times \F_{\kappa}^{\WW}$. Because of our assumption about $\VV$ and $\WW$, each $\WW$-term   also is a
$\VV$-term; hence, $\Theta$ generates a uniquely determined congruence  on $\F_{\kappa}^{\VV}$. More formally, there is a unique $\WW$-homomorphism
\[
u_{\kappa}\colon \F_{\kappa}^{\WW}\to U(\F_{\kappa}^{\VV})
\]
that extends the set-theoretic bijection between free generators. Writing
\[
u_{\kappa}^{2} \colon  (\F_{\kappa}^{\WW})^{2}\to U(\F_{\kappa}^{\VV})^{2}
\]
 for the product map $(x,y)\mapsto (u_{\kappa}(x),u_{\kappa}(y))$, let
 \[
 \widehat{u_k^2(\Theta)} \subseteq \F_{\kappa}^{\VV}\times \F_{\kappa}^{\VV}
 \]
 denote the $\VV$-congruence on $\F_{\kappa}^{\VV}$ generated by $u_\kappa^2(\Theta)$.
Then there  is a unique $\WW$-homomorphism
\[
B\cong \F_{\kappa}^{\WW}/\Theta
 \stackrel{\eta_{B}}{\longrightarrow}
 U(\F_{\kappa}^{\VV}/\widehat{u_\kappa^2(\Theta)})
\]
that extends the map that sends the $\Theta$-class of a free generator of $\F_{\kappa}^{\WW}$ to the $\widehat{u_\kappa^2(\Theta)}$-class of the corresponding generator of $\F_{\kappa}^{\VV}$.
(Here and throughout,~$\cong$ denotes the existence of an isomorphism). Now $\eta_{B}$ can be shown to be a universal arrow from $B$ to the functor $U$, in the sense of \cite[Definition on p.\ 55]{maclane}: for any $\VV$-algebra $A$ and any $\WW$-homomorphism $f\colon B \to U(A)$, there exists a unique $\VV$-homomorphism $f'\colon \F_{\kappa}^{\VV}/\widehat{u_\kappa^2(\Theta)} \to A$ such
that
${f=U(f')\circ\eta_{B}}$, 
i.e.,
such that the following diagram commutes.
\[
\xymatrix{
 & U(A) \ar@{<-}[d]^{f} \ar@{<--}[dl]_{U(f')} \\
 U(\F_{\kappa}^{\VV}/\widehat{u_\kappa^2(\Theta)}) \ar@{<-}[r]^{\ \ \ \ \ \ \eta_{B}} & B        }
\]
By \cite[Theorem IV.1.2]{maclane}, this universal property  of $\eta_{B}$ uniquely determines a functor $F\colon\WW\to \VV$ that is a left adjoint of $U$. The action of $F$ on a $\WW$-algebra $B$ is just
\begin{equation}\label{e:freeover}
F(B)=\F_{\kappa}^{\VV}/\widehat{u_\kappa^2(\Theta)} \ ,
\end{equation}
and one calls $F(B)$ the \emph{$\VV$-algebra freely generated
by the $\WW$-algebra $B$}; one also says that \emph{$F(B)$ is free over  $B$}.
(This terminology notwithstanding, note that neither $u_{\kappa}$ nor $\eta_{B}$ need  be injective, and that $B$ need not be isomorphic to a subalgebra of the $\WW$-reduct of $F(B)$.
Take for $\WW$ the variety of groups, and for $\VV$ the variety of Abelian groups. Then $F(B)$ is the Abelianization of $B$, and $\eta_{B}$ is a quotient map that need not be into).
As with all universal constructions, $F(B)$ is uniquely determined to within an isomorphism. Also, note that if $B$ is \emph{finitely presented}, i.e., if $\kappa$ is a finite integer and $\Theta$ is finitely generated, then so is $F(B)$, because in this case $\widehat{u_\kappa^2(\Theta)}$ is itself finitely generated by construction.

Given this broad setting, two questions  arise naturally.
\begin{enumerate}
\item[{\rm (I)}] {\it The description problem}. Given a $\WW$-algebra $B$, describe the structure of the $\VV$-algebra $F(B)$ as explicitly as possible in terms of the structure of $B$.
\item[{\rm (II)}] {\it The recognition problem}. Given a $\VV$-algebra $A$, find  conditions on the structure of $A$ that are necessary and sufficient for the existence of a   $\WW$-algebra $B$ such that $F(B)\cong A$.
\end{enumerate}
Let us point out that (II) is usually harder than (I), for quite general reasons. Indeed, (I) does not entail an existence question: the problem is to obtain a more transparent description of a \emph{given object}, namely, (\ref{e:freeover}). By contrast, (II) explicitly asks whether an object with a certain property \emph{exists}.

This paper is devoted to solving
problems (I) and (II) in case $\VV$ is the variety of {MV-algebras}, $\WW$ is the variety of Kleene algebras, and $B$ is finitely generated. When $\WW$ is, instead, the variety of distributive lattices, and $B$ is a finitely generated distributive lattice, a solution to (I) and (II) is given in \cite{marraarch}. Although our proofs are independent of \cite{marraarch}, the techniques used here build upon those employed in that paper. It may  also be of some interest to compare our treatment with the one in \cite{agm}, where (I) and (II) are solved in case $\VV$ is the subvariety of Heyting algebras generated by the linearly ordered algebras, $\WW$ is the variety of distributive lattices, and $B$ is finitely generated. In all cases, the availability of an efficient duality theory for  (relevant classes of)  the algebras at hand plays a crucial r\^{o}le.

The standard reference for the elementary theory of MV-algebras is \cite{cdm},
 whereas \cite{mundicibis} deals
with advanced topics.
An \emph{MV-algebra} is an algebraic structure $(M,\oplus,\neg,0)$, where $0\in M$
is a constant, $\neg$ is a unary operation satisfying $\neg\neg x=x$, $\oplus$ is a unary operation making $(M,\oplus,0)$
a commutative monoid, the element $1$ defined as $\neg 0$ satisfies $x\oplus1=1$, and the law
\begin{align}
 \neg(\neg x \oplus y)\oplus y = \neg(\neg y \oplus x)\oplus x\tag{MV}\label{mvlaw}
\end{align}
holds.  Any MV-algebra has an underlying structure of
distributive lattice bounded below by $0$ and above by $1$. Joins are defined as $x \vee y = \neg(\neg x \oplus  y)\oplus y$.
Thus, the characteristic law (\ref{mvlaw}) states
that $x\vee y=y\vee x$.
Meets are defined by the De Morgan condition  $x \wedge y = \neg (\neg x \vee \neg y)$.
The De Morgan dual of $\oplus$, on the other hand, is the operation traditionally denoted
\[
x \odot y = \neg (\neg x \oplus \neg y)
\]
which, like $\oplus$, is not idempotent. Boolean algebras are
those MV-algebras that are idempotent, meaning that $x\oplus x = x$ holds, or equivalently, that $x\odot x = x$ holds, or equivalently, that satisfy the {\it tertium non datur} law $x\vee\neg x=1$.

Let us further recall that a \emph{De Morgan algebra} is an algebraic structure $(A,\vee,\wedge,\neg,0,1)$ such that $(A,\vee,\wedge,0,1)$ is a distributive lattice with bottom element $0$ and top element $1$, $\neg$ is a unary operation satisfying $\neg\neg x=x$ and $\neg 0 = 1$, and the De Morgan law
\[\label{e:M}\tag{DM}
\neg(x\wedge y)=\neg x\vee\neg y
\]
holds.
Such an algebra is a \emph{Kleene algebra} if it satisfies the additional  equational law
\[\label{e:K}\tag{K}
(x\wedge\neg x) \vee (y\vee\neg y)  = y\vee\neg y 
.
\]
We write $\MM$ and $\KK$ for the varieties of MV-algebras and   Kleene algebras respectively. The real unit interval $[0,1]\subseteq \R$ can be made into an MV-algebra with neutral element $0$
by defining $x\oplus y = \min{\{x+y,1\}}$ and $\neg x=1-x$. The underlying lattice order of this MV-algebra coincides with
the order that $[0,1]$ inherits from the real numbers. When we
refer to  $[0,1]$ as an MV-algebra, we always mean the  structure just described. \emph{Chang's completeness theorem}  states that $[0,1]$ generates the variety $\MM$ \cite[2.5.3]{cdm}. It is then easily proved that (\ref{e:M}) and (\ref{e:K}) hold in all MV-algebras, by checking that they hold in $[0,1]$. We thus obtain a forgetful functor
$U\colon \MM \to \KK$
that assigns to each MV-algebra  its underlying Kleene algebra. We further let $F\colon \KK \to \MM$ be the left adjoint to $U$, as described above.  The three-element set
\[\textstyle 
K=\left\{0,\frac{1}{2},1\right\}
,
\]
endowed with its natural order, and equipped with the negation operation that  fixes $\frac{1}{2}$ and swaps $0$ with $1$, is a Kleene algebra.
In \cite[Lemma 2]{kalman},  Kalman showed  that $K$ and its only proper subalgebra, the two-element Boolean algebra $\{0,1\}$, are precisely the  non-singleton subdirectly irreducible Kleene algebras; therefore, $K$ generates the variety $\KK$. It follows at once that $\KK$ is locally finite (=
finitely generated Kleene algebras are finite), hence the three classes of finite, finitely presented, and finitely  generated Kleene algebras are coextensive. From now on we therefore deal exclusively with finite Kleene algebras. For further background and references on Kleene and De Morgan algebras please see \cite[Chapter XI]{bd}.

Our solution to problems (I) and (II) relies on the fact that an MV-algebra $A$  is  finitely presented  if and only if there is $n\in\{1,2,\ldots\}$ and a rational polyhedron $P\subseteq [0,1]^n$  such that $A$ isomorphic to the MV-algebra $\McN{(P)}$ of continuous piecewise linear functions $f\colon P \to [0,1]$ whose linear pieces are expressible as linear polynomials with integer coefficients (see \cite{mundicibis}  or Lemma~\ref{Th:DualAsSolutionSetMV} below).
To solve problem (I), given a Kleene algebra $B$ we determine a rational polyhedron
$P$ such that $F(B)$ is the finitely presented algebra $\McN{(P)}$. To solve problem (II), we first determine conditions that identify the rational polyhedra $P$ such that $\McN{(P)}$ correspond to $F(B)$ for some Kleene algebra $B$, and then we translate these into algebraic terms.

To develop this working strategy we need a number of background results.
First, we heavily exploit the finite fragment of the natural duality for Kleene algebras due to Davey and Werner \cite{DW} (see also \cite{CD}). The needed facts are collected in
Section~\ref{s:natural}. The object that is dual to a finite Kleene algebra is called a Kleene space: it is a finite  partially ordered set (= poset) endowed with a distinguished subset $M$ of its collection of maximal elements, and an additional binary relation $R$ satisfying appropriate conditions; cf.~
Definition \ref{d:kleenespace}. For example, consider the  Kleene chain (= totally ordered set) $C=(\{0,a,b,c,d,1\}, \min, \max, \neg, 0, 1)$ whose Hasse diagram is depicted in Figure \ref{fig:chain}.a,  and whose negation
is determined by $\neg a = d$ and $\neg b = c$.
It will transpire that its dual Kleene space is the three-element chain $D(C)$,  with $M$ the set containing its top element, and $R$ the total relation; see Figure \ref{fig:chain}.b.
\begin{figure}[h]
\centering
\subfigure[\text{The Kleene algebra $C$.}]
{\includegraphics[height=38mm, width=43mm]{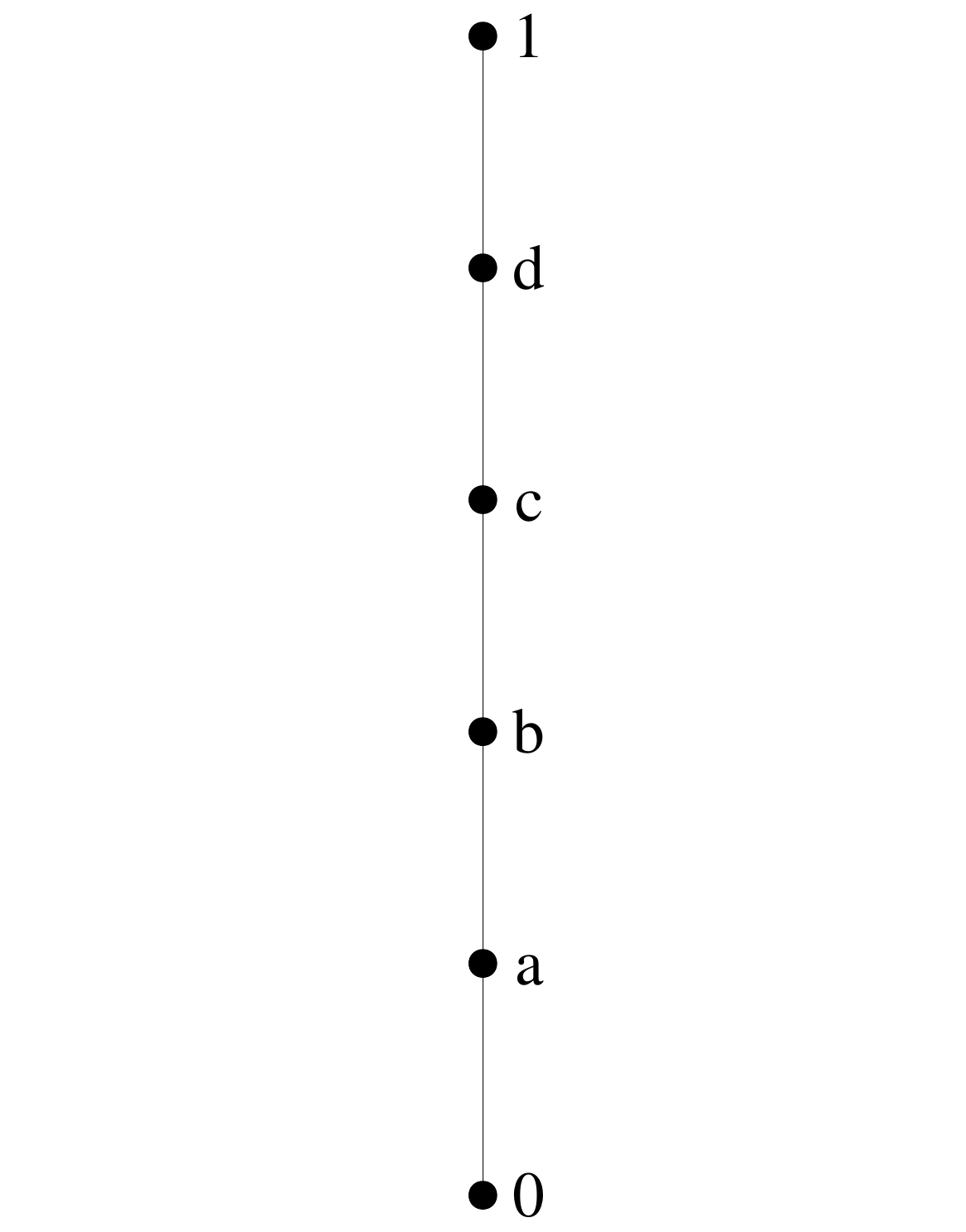}}
\qquad
\subfigure[\text{The Kleene space $D(C)$.}]
{\includegraphics[height=38mm, width=37mm]{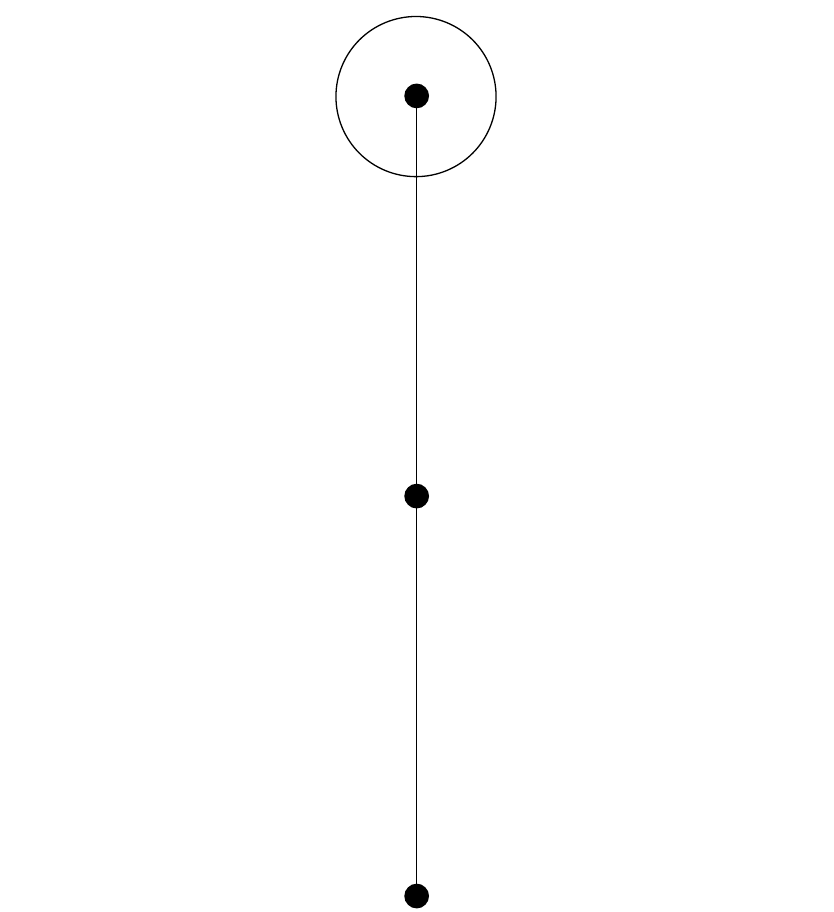}}
\caption{The natural duality for Kleene algebras.}
\label{fig:chain}
\end{figure}
In Section~\ref{s:description} we show how to construct an MV-algebra out of the dual  Kleene space of a finite Kleene algebra.
To any finite  poset a classical construction associates the abstract simplicial complex of its chains. This is known as the order complex, or nerve, of the poset.
The nerve $\N{(D(C))}$ of the Kleene space $D(C)$ of $C$, for instance, is ``the triangle'' shown in Figure \ref{fig:exe}.a; its ``vertices''  (= singleton chains) correspond to elements of the Kleene space $D(C)$. It turns out that
the nerve of the Kleene space of a finite Kleene algebra carries a natural weighing of its vertices induced by  the distinguished subset
$M$ of its maximal elements: vertices in $M$ are weighed $1$; vertices not in $M$ are weighed $2$. The resulting distribution of weights in our current example is shown  in Figure \ref{fig:exe}.a, too.

\begin{figure}[h]
\centering
\subfigure[\text{The nerve $\N(D(C))$.}]
{\includegraphics[scale=.4]{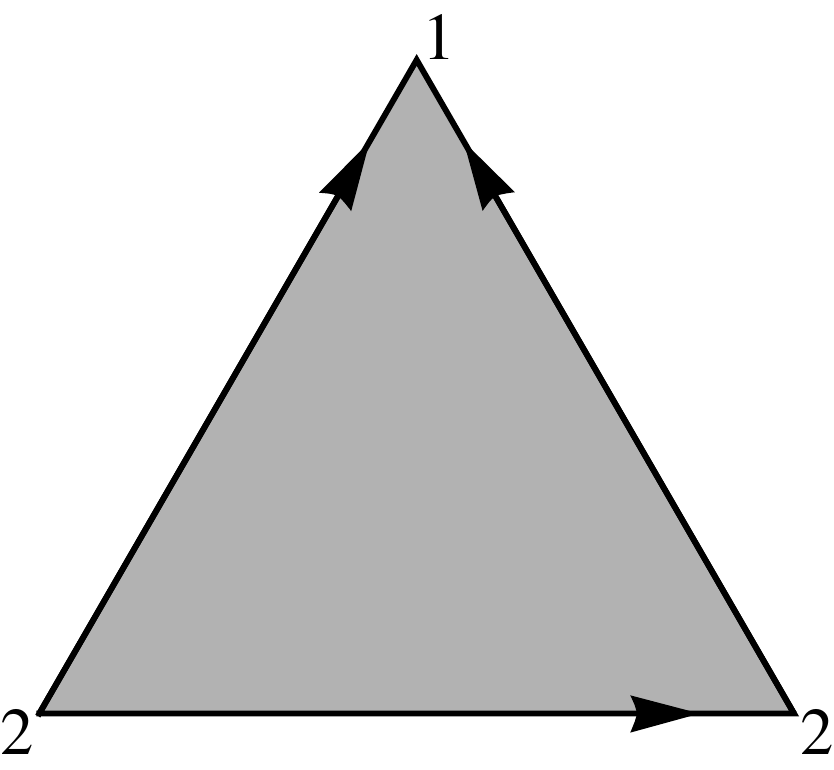}}
\qquad\qquad\qquad\qquad
\subfigure[\text{A  triangle $T\subseteq[0,1]^{3}$.}]
{\includegraphics[scale=.4]{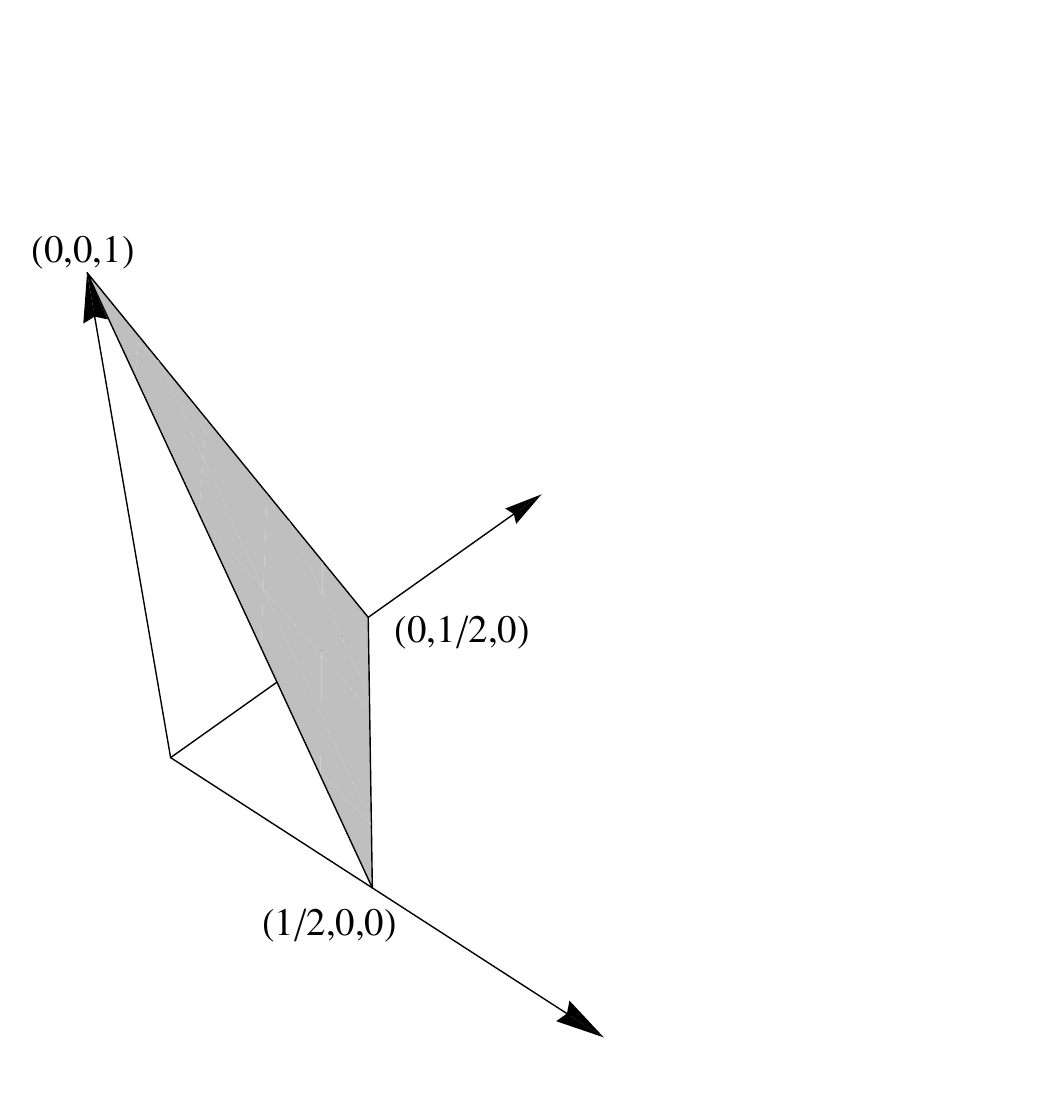}}
\caption{A weighted nerve and its geometric realization as a rational polyhedron.}
\label{fig:exe}
\end{figure}

 To a weighted abstract simplicial complex  we associate its geometric realization as a rational polyhedron   in some finite-dimensional Euclidean vector space $\R^n$, where $n$
 is the number of vertices of the abstract complex.
  Figure \ref{fig:exe}.b depicts the geometric realization of the weighted nerve  $\N{(D(C))}$, a triangle $T$ in $[0,1]^{3}\subseteq \R^{3}$. The geometric realization is defined in such a way
(cf.~(\ref{eq:nerve}) in Section~\ref{s:description})  that each weight of
$\N{(D(C))}$ agrees with the denominator of the vector of rational coordinates of the corresponding vertex of $T$.
(See (\ref{eq:den}) in Section~\ref{s:description} for the notion of denominator of a rational vector.)
 Finally, to the rational polyhedron $T$ we associate the MV-algebra $\McN{(T)}$. Our Theorem I, proved in Section~\ref{s:description}, shows that $\McN{(T)}$ is the MV-algebra freely generated by the Kleene chain $C$ shown in Figure \ref{fig:chain}.a. Quite generally, Theorem I states that
\emph{the MV-algebra $F(B)$ freely generated by a finite Kleene algebra $B$ is the MV-algebra $\McN{(P)}$ of $\Z$-maps on the rational polyhedron $P$ that is the geometric realization of the weighted nerve $\N(D(B))$ of the Kleene space $D(B)$ dual to $B$}. This solves the description problem.

The solution to the recognition problem requires
one further piece of
background on MV-algebras,
namely, the theory of bases. The needed facts are recalled in Section~\ref{s:MV-duality}. Bases are generating sets of (necessarily finitely presented) MV-algebras characterised by very special, purely algebraic properties; see Definition \ref{d:basis} in Section~\ref{s:MV-duality}.  For example, Figure \ref{fig:schauder} shows a triangulation of $[0,1]^{2}$---called  a Kleene triangulation  in Definition \ref{d:kleenet}---along with some of the pyramidal functions at its vertices. (The height of the pyramid at the vertex~$v$ equals $\frac{1}{\den{v}}$, where the denominator $\den{v}$ is defined as in (\ref{eq:den}) below.) Such pyramidal functions are known as Schauder hats of the MV-algebra
$\McN{([0,1]^{2})}$. The set of all Schauder hats over the pictured triangulation turns out to be a basis of $\McN{([0,1]^{2})}$; see Lemma \ref{lemma:schauder-basis}. The defining properties of bases are to the effect that a basis of an MV-algebra $A$ completely encodes a rational polyhedron $P$ in $\R^{d}$,  unique to within an appropriate notion of homeomorphism, such that $A$ may be represented as the MV-algebra $\McN{(P)}$ of $\Z$-maps on $P$. Special classes of bases, then, correspond to polyhedra $P$ admitting  triangulations with special properties. The recognition problem (II) can be solved because we are able to characterise in this manner, via bases, those special triangulations that come from taking the weighted nerve of a Kleene space. Indeed, in Definition \ref{d:kleenebasis} of Section~\ref{s:recognition} we isolate a class of bases that we call Kleene bases; they are defined by elementary combinatorial properties. Theorem II, proved in Section~\ref{s:recognition}, asserts that \emph{an MV-algebra is free over some finite Kleene algebra if, and only if, it possesses a Kleene basis}.  This solves the recognition problem.

\begin{figure}[h]
\centering
{\includegraphics[scale=.6]{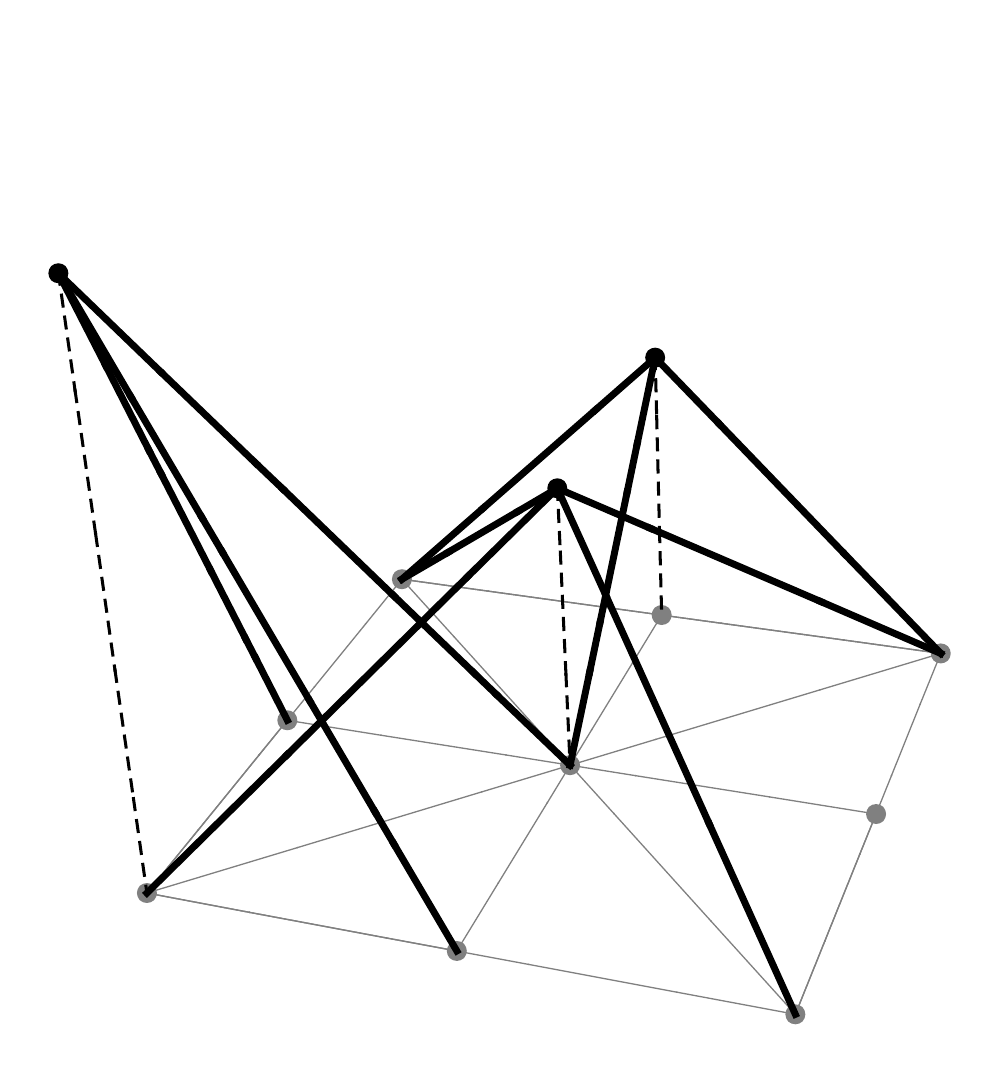}}
\caption{The Kleene  triangulation $\mathcal{S}_{2}$ of $[0,1]^2$, together with some Schauder hats.}
\label{fig:schauder}
\end{figure}

\section{Natural duality for Kleene algebras.}\label{s:natural}
A \emph{subposet} of a poset $(P,\leq)$ is a subset of $P$ endowed with the order inherited from $\leq$ by restriction. An order-preserving map $f\colon P \to Q$   between posets is an \emph{order-embedding} if it is injective and
reflects the order: $f(p_1)\leq f(p_2)$ implies $p_1\leq p_2$ for each $p_1,p_2\in P$. A \emph{chain} is a totally ordered set. If $P$ is a poset, $\max{P}$ denotes the set of maximal elements of $P$. The category of posets (with morphisms the order-preserving maps) has products; the $n$-fold product of $P$ with itself is denoted $P^{n}$, and consists of the set of $n$-tuples of elements of $P$ endowed with the coordinatewise order.

As  mentioned, the variety $\KK$ of Kleene algebras is generated by the algebra
$$
\textstyle K=(\{0,\m,1\},\wedge,\vee,\neg,0,1)
,
$$
where $(\{0,\m,1\},\wedge,\vee,0,1)$ is the bounded distributive lattice whose order is $0\leq \m \leq 1$ (a chain), and $\neg\colon \{0,\m,1\}\to \{0,\m,1\}$ is defined by $\neg(0)=1$, $\neg(1)=0$ and $\neg(\m)=\m$.
Moreover, every Kleene algebra is isomorphic to a subalgebra of products of the algebra $K$, that is, $\KK=\mathbb{ISP}{(K)}$;  see \cite{kalman}.
Building on these results, \cite{DW}  develops a natural duality for  Kleene algebras. (See also \cite[Theorem 4.3.10]{CD}.)
We recall here the facts about  this natural duality that we will need in the sequel; throughout, we are only concerned with finite Kleene algebras.
(See \cite{CD} for background on natural dualities.)

\begin{definition}\label{d:kleenespace}
  A structure $(W,\leq,R,M)$ is called a {\em \textup{(}finite\textup{)} Kleene space} if $(W,\leq)$ is a finite poset, $M\subseteq \max{W}$,  $R\subseteq W^{2}$, and the following hold.
  \begin{itemize}
    \item[(i)] $(x,x)\in R$;
    \item[(ii)] $(x,y)\in R$ and  $x\in M$ imply $y\leq x$;
    \item[(iii)] $(x,y)\in R$ and $z\leq y$ imply $(z,x)\in R$;
  \end{itemize}
  for each $x,y,z\in W$. Further,  a \emph{morphism of Kleene spaces} $(W,\leq,R,M)$ and $(W',\leq',R',M')$ is a function $f\colon W\to W'$  that is order-preserving, preserves $R$ (i.e., $(x,y)\in R$ implies $(f(x),f(y))\in R'$), and satisfies $f(M)\subseteq M'$.
  We write $\KS$ for the category of Kleene spaces and their  morphisms.
\end{definition}

The natural duality between Kleene algebras and Kleene spaces is determined by the Kleene space $\widetilde{K}=(\{0,\m,1\},\preceq,\sim, \{0,1\})$ where $\preceq$ is the order whose Hasse diagram is depicted in Figure \ref{Fig:V}, and
$\sim$ is a binary relation defined by: 
$(x,y)\in{\sim}$
iff

\begin{figure}[h]
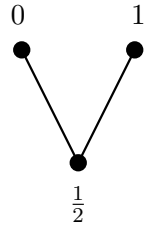

\centering
\begin{pgfpicture}{19.00mm}{83.14mm}{41.00mm}{116.14mm}
\pgfsetxvec{\pgfpoint{1.00mm}{0mm}}
\pgfsetyvec{\pgfpoint{0mm}{1.00mm}}
\color[rgb]{0,0,0}\pgfsetlinewidth{0.30mm}\pgfsetdash{}{0mm}
\pgfcircle[fill]{\pgfxy(22.49,107.51)}{1.00mm}
\pgfcircle[stroke]{\pgfxy(22.49,107.51)}{1.00mm}
\pgfcircle[fill]{\pgfxy(30.00,92.49)}{1.00mm}
\pgfcircle[stroke]{\pgfxy(30.00,92.49)}{1.00mm}
\pgfcircle[fill]{\pgfxy(37.51,107.51)}{1.00mm}
\pgfcircle[stroke]{\pgfxy(37.51,107.51)}{1.00mm}
\pgfputat{\pgfxy(22.00,111.00)}{\pgfbox[bottom,left]{\fontsize{10.95}{13.66}\selectfont \makebox[0pt]{$0$}}}
\pgfputat{\pgfxy(38.00,111.00)}{\pgfbox[bottom,left]{\fontsize{10.95}{13.66}\selectfont \makebox[0pt]{$1$}}}
\pgfputat{\pgfxy(30.00,86.00)}{\pgfbox[bottom,left]{\fontsize{10.95}{13.66}\selectfont \makebox[0pt]{$\m$}}}
\pgfmoveto{\pgfxy(22.49,107.51)}\pgflineto{\pgfxy(30.00,92.49)}\pgfstroke
\pgfmoveto{\pgfxy(30.00,92.49)}\pgflineto{\pgfxy(37.51,107.51)}\pgfstroke
\end{pgfpicture}%
\caption{The Kleene space $\widetilde{K}$.}\label{Fig:V}
\end{figure}
Let
\[
\widetilde{K}^{n}=\left(\{0,\m,1\}^{n},
 \preceq_n,
 \sim_n,
 \{0,1\}^{n}\right)
\]
be the Kleene space whose underlying poset is $\{0,\m,1\}^{n}$, where the relations $\preceq_n$ and $\sim_n$ are defined componentwise from the relations $\preceq$ and $\sim$ on $\widetilde{K}$, respectively. Thus, using infix notation, $(x_{1},\ldots,x_{n})\preceq_{n}(y_{1},\ldots,y_{n})$ iff $x_{i}\preceq y_{i}$ for each $i=1,\ldots, n$, where $x_{i},y_{i}\in \{0,\m,1\}$; similarly for $\sim_{n}$.
The Kleene space $\widetilde{K}^{2}$ is depicted in Figure \ref{fig:k2}. Here, nodes are labelled $1$ if they belong to~$\{0,1\}^{2}$, and $2$ otherwise.
(The reason we label nodes  this way will be explained later, cf.~
 Definition \ref{d:asspoly} and Section~\ref{s:recognition}). The relation $\sim_{2}$ is not  represented.
\begin{figure}[h]
\centering
{\includegraphics[scale=0.6]{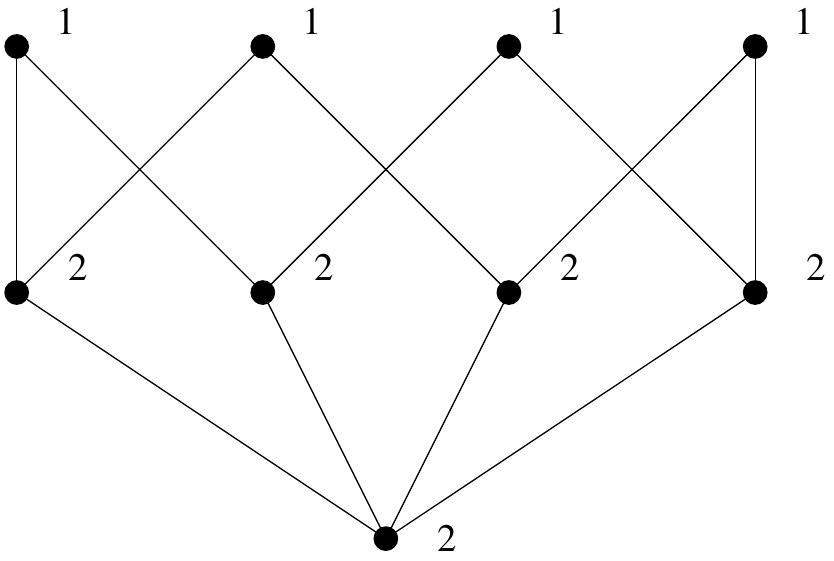}}
\caption{The underlying poset of the Kleene space $\widetilde{K}^{2}$, with labelled nodes.}
\label{fig:k2}
\end{figure}

 Let us write $\KKf$ for the category of finite Kleene algebras and their homomorphisms. Now the functor $D\colon \KKf\to \KS$ is defined as follows. For each finite Kleene algebra $B$, let
$$
D(B)=\{h\colon B\to K\mid h\mbox{ is a homomorphism}\}\subseteq \widetilde{K}^{B}
,
$$
endowed with the structure of Kleene space inherited from $\widetilde{K}^{B}$ by restriction.
For each homomorphism $f\colon B\to C$ of Kleene algebras, let
$D(f)$ be the map from $D(C)$ to $D(B)$
 defined by
$$
D(f)(h)=h\circ f
$$
for each $h\in D(C)$.

This functor $D$ determines a categorical equivalence between $\KKf$ and the opposite of the category of Kleene spaces, $\KS^{\rm op}$. The adjoint to $D$ is the functor $E\colon \KS\to \KKf$ defined as follows. For each  Kleene space $X$, let
$$
E(X)=\{h\colon X\to \widetilde{K}\mid h\mbox{ is a morphism of Kleene spaces}\}\subseteq K^{X}
,
$$
endowed with the structure of Kleene subalgebra of ${K}^{X}$.
For each morphism $f\colon X\to Y$ of Kleene spaces, let $E(f)\colon E(Y)\to E(X)$ be defined by
$$
E(f)(h)=h\circ f
$$
for each $h\in E(Y)$. An application of \cite[Proposition 2.2.3]{CD} yields:
\begin{lemma}\label{l:freeK}
For each  integer $n\geq 1$, the Kleene algebra $E(\widetilde{K}^{n})$ is the free  Kleene algebra over $n$ generators.  A free generating set of $E(\widetilde{K}^{n})$ is given by the projection maps
$\rho_i\colon\{0,\m,1\}^{n}\to \{0,\m,1\}$,  $i=1,\ldots,n$.\qed
\end{lemma}

Given any subset $\Theta\subseteq E(\widetilde{K}^{n})^{2} $ we define:
\[
\Sol_{\KK}{(\Theta)}=\{v\in \{0,\m,1\}^{n}\mid f(v)=g(v) \mbox{ for each }(f,g)\in \Theta\}.
\]
It is an easy consequence of the natural duality above that any subset of $\{0,\m,1\}^{n}$ is the solution set of equations in the language of Kleene algebras.
\begin{lemma}\label{Th:DualAsSolutionSetKleene}
Let $W\subseteq \{0,\m,1\}^{n}$ be any subset. Then $(W,\preceq,\sim,M)$ is a subobject of the Kleene space
$\widetilde{K}^{n}=\left(\{0,\m,1\}^{n},
 \preceq_n,
\sim_n,
  \{0,1\}^{n}\right)$ if the binary relations $\preceq$, $\sim$, and the unary relation $M$ are defined by restriction
from~$\preceq_{n}$,~$\sim_{n}$, and $\{0,1\}^{n}$, respectively.  Writing $\iota\colon (W,\preceq,\sim,M)\hookrightarrow \widetilde{K}^{n}$ for the corresponding monomorphism of Kleene spaces, we have
\[
W=
\Sol_{\KK}{(\Ker{E(\iota)})}
,
\]
where  $\Ker{E(\iota)}=\{(f,g)\in E(\widetilde{K}^n)^2\mid E(\iota)(f)=E(\iota)(g)\}$.
\end{lemma}
\begin{proof}
Observe first
that if $(f,g)\in E(\widetilde{K}^{n})^{2}$, then $(f,g)\in\Ker{E(\iota)}$
iff $f\circ \iota=g\circ \iota$ iff $f(v)=g(v)$ for each $v\in W$.
Thus  $W\subseteq\Sol_{\KK}{(\Ker{E(\iota)})}$.

Assume now that $x\not\in W$. Define $f\colon \{0,\m,1\}^{n}\to \{0,\m,1\}$ by
\[
f(y)=\begin{cases}
 1& \mbox{if there exists }z\in W\mbox{ such that }z\preceq_{n} y\mbox{ and }x\preceq_{n} z;\\
0&\mbox{if }y \in \{0,1\}^{n}\mbox{ and  there does not exist }z\in W\mbox{such}\\ &\mbox{ that }z\preceq_{n} y\mbox{ and }x\preceq_{n}z ;\\
\m&\mbox{ otherwise.}\\
\end{cases}
\]
 Further, define $g\colon \{0,\m,1\}^{n}\to \{0,\m,1\}$ by
 \[
g(z)=\begin{cases}
1&\mbox{if }x\preceq_{n} z;\\
0&\mbox{if }z\in \{0,1\}^{n}\mbox{ and }x\not\preceq_{n} z ;\\
\m&\mbox{otherwise.}\\
\end{cases}
\]
A  simple verification shows that $f,g\in E(\widetilde{K}^{n})$. Since $x\not\in W$, $f(x)\in\{\m,0\}$ and $g(x)=1$. It follows that $x\not\in\Sol_{\KK}{(\Ker{E(\iota)})}$, as was to be shown.
\end{proof}

\section{The  companion polyhedron of a Kleene space: solution to the description problem.}\label{s:description}
We now show how to canonically construct an MV-algebra out of a Kleene space. It will transpire that this provides a solution to the description problem.

An \emph{abstract simplicial complex} over a finite set $V$ is a family $\SC$ of subsets of $V$ that is closed under the operation of taking subsets, and includes all singleton subsets.
The elements of $\SC$ of cardinality $k$ are called (\emph{abstract}) \emph{$(k-1)$-simplices};  the $0$-simplices are called \emph{vertices} of the complex. An \emph{isomorphism} of abstract simplicial complexes $\SC$ and $\SC'$ over the sets $V$ and $V'$, respectively, is a bijection $f\colon V \to V'$ that carries simplices to simplices, i.e., is such that $\{v_{1},\ldots,v_{u}\}\in \SC$  if, and only if, $\{f(v_{1}),\ldots,f(v_{u})\}\in \SC'$.

A \emph{weighted abstract simplicial complex} is a pair $(\SC,\omega)$, where $\SC$ is an abstract simplicial complex over $V$,
and $\omega$ is a weight function $\omega \colon V \to \{1,2,\dots\}$.
 An \emph{isomorphism} of abstract weighted simplicial complexes $(\SC,\omega)$ and $(\SC',\omega')$ over the sets $V$ and $V'$, respectively, is an isomorphism $f\colon V \to V'$ of $\SC$ and $\SC'$ that preserves weights, i.e., satisfies $\omega'\circ f=\omega$.

Abstract simplicial complexes provide combinatorial descriptions of polyhedra. For background on polyhedral geometry, please see
\cite{alexandrov, rs};  we will recall the basic notions. Let us consider the finite-dimensional real vector space
$\R^d$.
A \emph{convex combination} of a finite set of vectors $v_1,\ldots,v_u \in \R^d$
is any vector of the form
$\lambda_1v_1+\cdots+\lambda_uv_u$, for non-negative real numbers $\lambda_i
\geq 0$ satisfying $\sum_{i=1}^u\lambda_i=1$.
If $S\subseteq \R^d$ is any subset, we let $\conv{S}$ denote the \emph{convex
hull} of $S$, i.e., the collection of all convex combinations
of finite sets of vectors $v_1,\ldots,v_u \in S$. A \emph{polytope}  is any
subset of $\R^d$ of the form
$\conv{S}$, for some finite $S\subseteq \R^d$, and a (\emph{compact})
\emph{polyhedron}  is a  union of finitely many polytopes
in $\R^d$.
(Polyhedra need not be convex, in accordance with the  usage in \cite{rs}.)

To each abstract simplicial complex $\SC$ over the set $V=\{v_1,\ldots,v_d\}$ of vertices one associates a  polyhedron contained
in $\R^{d}$. Let $\{e_1,\ldots,e_{d}\}$ be the standard  basis of
$\R^{d}$.
For each abstract simplex $S=\{v_{i_1},\ldots, v_{i_{u}}\}\in \SC$, consider the polytope
\[
\widehat{S}=\conv{\{e_{i_1},\ldots, e_{i_{u}}\}}\subseteq\R^{d}
.
\]
Then the polyhedron
\[
P_{\SC}=\bigcup_{S\in\SC} \widehat{S}
\]
is called the  \emph{geometric realization} of $\SC$. We  adapt this construction to weighted complexes by rescaling the unit vector $e_{i}$ by a factor of $\frac{1}{w}$, if $w$ is the weight involved. Let $(\SC,\omega)$ be a weighted abstract simplicial complex over  the set
$V=\{v_1,\ldots,v_d\}$. For each abstract simplex $S=\{v_{i_1},\ldots, v_{i_{u}}\}\in \SC$, consider the polytope
\[
\overline{S}=\conv{\left\{\frac{e_{i_1}}{\omega(v_{i_1})},\ldots, \frac{e_{i_{u}}}{\omega(v_{i_u})}\right\}}\subseteq\R^{d}
.
\]
We call the polyhedron
\begin{align}\label{eq:nerve}
P_{\SC}^{\omega}=\bigcup_{S\in\SC} \overline{S}
\end{align}
the  \emph{geometric realization} of $(\SC,\omega)$.

A standard construction produces an abstract simplicial complex from a  poset. The \emph{nerve of a \textup{(}finite\textup{)} poset} $O$, denoted $\N{(O)}$, is the collection of
 totally ordered subsets of $O$.
 It is then clear that $\N{(O)}$ is an abstract simplicial complex; it is alternatively called the \emph{order complex of $O$}. For background and further references we refer to \cite{bjoerner}.
\begin{definition}\label{d:asspoly} Let $(W,\leq,R,M)$ be a Kleene space. Its \emph{associated weighted abstract simplicial complex} is defined as $(\N{(W)},\omega)$, where, for each
${w\in W}$, $\omega(w)=1$ if $w\in M$, and $\omega(w)=2$ otherwise. Further,
 its \emph{companion polyhedron} is defined as the geometric realization of $(\N{(W)},\omega)$ as in (\ref{eq:nerve}).
 (Note $(\N{(W)},\omega)$ does not depend on $R$.)
\end{definition}
Having constructed a polyhedron out of a Kleene space via its weighted nerve, we construct an MV-algebra out of the polyhedron. For background on the extensive interaction between MV-algebras and polyhedra, please see~\cite{mundicibis}.
Throughout, the adjective `linear' is to be understood as `affine linear'. A function $f \colon \R^d \to \R$ is \emph{piecewise linear} if it is continuous
(with respect to the
Euclidean topology on $\R^d$ and $\R$), and there is a finite set of
linear
functions $l_1,\ldots,l_u$ such that for each $x \in \R^d$ one has $f(x)=l_i(x)$
for some choice of $i=1,\ldots,u$.
If, moreover, each $l_i$ can be written as a linear polynomial with integer
coefficients, then $f$ is a  \emph{$\Z$-map}.
If $X\subseteq [0,1]^{d}\subseteq \R^{d}$ is a subspace contained in the unit cube $[0,1]^{d}$, a \emph{$\Z$-map} on $X$ is a function
$f\colon X \to \R$ that agrees with a restriction of some $\Z$-map $\R^{d}\to\R$ to $X$. We set
\begin{align}\label{eq:zmaps}
\McN{(X)}=\{f \colon X \to [0,1]
 \mid
 f \text{ is a $\Z$-map}\}
.
\end{align}
It is an exercise
to prove that  $\McN{(X)}$ is an MV-algebra when endowed  with operations defined pointwise from the standard MV-algebra $[0,1]$.
 Explicitly,
one defines, $(f\oplus g)(x)=\min{\{f(x)+g(x),1\}}$ and $(\neg f)(x)=1-f(x)$  for each $f,g\in\McN{(X)}$ and $x \in X$; the neutral element for $\oplus$ is the function constantly equal to $0$ on $X$.  We call $\McN{(X)}$ the \emph{MV-algebra of $\Z$-maps on $X$}.
\begin{theoremI}[The description problem]\label{t:main1}For each finite Kleene algebra $B$, the MV-algebra $F(B)$ freely generated by $B$ is isomorphic to the MV-algebra of $\Z$-maps on the  companion polyhedron of the Kleene space dual to $B$. In symbols, let $D(B)=
(W,\leq,R,M)$ be the Kleene space dual to $B$,  let $(\N{(W)}, \omega)$ be its associated weighted abstract simplicial complex, and let $P_{\N{(W)}}^{\omega}$ be its companion polyhedron  as in \textup{(\ref{eq:nerve})}. Then
\[
F(B) \cong \McN{(P_{\N{(W)}}^{\omega})}
.
\]
\end{theoremI}

 The proof  of Theorem I will require several lemmas. For an integer $n\geq 1$, we write $[0,1]^{n}$ to denote the product of the unit interval $[0,1]\subseteq \R$, endowed with the product topology. (The topology on $[0,1]$ is always the Euclidean subspace  topology inherited from $\R$.) We further write $\McN_{n}$ as a shorthand for $\McN{([0,1]^{n})}$.
\begin{lemma}[The Chang-McNaughton Theorem]\label{l:freerep}For each integer $n\geq 1$, the MV-algebra $\McN_{n}$ is the free MV-algebra over $n$ generators. A free generating set of $\McN_{n}$ is given by the projection maps $\pi_{i}\colon [0,1]^{n}\to [0,1]$, $i=1,\ldots,n$.
\end{lemma}
\begin{proof}This is \cite[9.1.5]{cdm}.
\end{proof}
For an integer $n\geq 1$, given any subset $\Theta\subseteq \McN^{2}_{n}$ we define:
\[
\Sol_{\MM}{(\Theta)}=\{v\in [0,1]^{n} \mid f(v)=g(v) \mbox{ for each }(f,g)\in \Theta\}.
\]
\begin{lemma}[The Hay-W\'ojcicki Theorem]\label{Th:DualAsSolutionSetMV}Fix an integer $n\geq 1$, and a finite subset $\Theta\subseteq \McN^{2}_{n}$. If $\widehat{\Theta}$ denotes the congruence on $\McN_{n}$ generated by $\Theta$, we have:
\[
\McN_{n}/\widehat{\Theta}\cong \McN{(\Sol_{\MM}{(\Theta)})}
.
\]
\end{lemma}
\begin{proof}This is, in essence, \cite[3.6.9]{cdm}.
\end{proof}
\begin{remarknonum}Caution: if one drops the finiteness condition from the statement of the Hay-W\'ojcicki Theorem, the resulting assertion is false, because there exist finitely generated non-semisimple MV-algebras. See  \cite[3.6 and 4.6]{cdm} for more details.\qed\end{remarknonum}

Given an integer $n\geq 1$, let $\rho_i\colon \{0,\m,1\}^{n}\to \{0,\m,1\}$ and $\pi_i\colon [0,1]^{n}\to [0,1]$ denote the respective projection maps onto the $i^{\rm th}$ coordinate, $i=1,\ldots,n$. By Lemma \ref{l:freeK}, $\{\rho_i\}_{i=1}^{n}$ generates the Kleene algebra $E(\widetilde{K}^{n})$ freely. Hence there exists a unique Kleene homomorphism
\begin{align*}
u_n\colon E(\widetilde{K}^n)\longrightarrow \McN_{n}
\end{align*}
that extends the assignment
\[
\rho_{i}\longmapsto \pi_{i} 
, 
i=1,\ldots,n.
\]
Let us write
\begin{align}\label{eq:un2}
u_n^{2}\colon E(\widetilde{K}^n)^{2} \longrightarrow \McN_{n}^{2}
\end{align}
for the product map of $u_{n}$ with itself given by
\[
(f,g)\in E(\widetilde{K}^n)^2
\longmapsto 
 (u_{n}(f),u_{n}(g)) \in\McN_{n}^{2}
.
\]
Now $K=\{0,\m,1\}$ is a Kleene subalgebra of the Kleene reduct of the standard MV-algebra $[0,1]$, and, by \cite[Lemma 2]{kalman}, $K$ generates the variety $\KK$. This shows that $u_n$ is a one-one map, and proves the following result that allows us to express  $\Sol_{\KK}{(\cdot)}$ in terms of $\Sol_{\MM}{(\cdot)}$:
\begin{lemma}\label{Th:KleeneMVSolutions}Fix an integer $n\geq 1$.
For any subset $\Theta\subseteq E(\widetilde{K}^{n})^{2}$ we have
\[
\textstyle\Sol_{\KK}{(\Theta)}=
\Sol_{\MM}{(u_n^{2}(\Theta))}\cap\left\{0,\m,1\right\}^{n}
,
\]
where $u_{n}^{2}$ is as in \textup{(\ref{eq:un2})}.\qed
\end{lemma}
%
 Conversely, in Lemma \ref{Cor:SimplexSolution} we will present a way to express $\Sol_{\MM}{(\cdot)}$ in terms of $\Sol_{\KK}{(\cdot)}$. To do so we need to introduce some further background on rational polyhedral geometry.

A polytope  that may be written as $\sigma=\conv{S}$, for
$S=\{v_0,v_1,\ldots,v_u\}$ a finite set of affinely independent
vectors in $\R^{d}$, is a \emph{\textup{(}$u$-dimensional\textup{)} simplex}, or
a \emph{$u$-simplex} for short; $S$ is then
the (uniquely determined) set of
\emph{vertices} of~$\sigma$, denoted $\ver{\sigma}$. The simplex $\sigma$ is \emph{rational} if $\ver{\sigma}
\subseteq \Q^d$.
A \emph{\textup{(}$w$-dimensional\textup{)} face} of $\sigma$ is any simplex of
the form $\conv{S'}$, for $S'\subseteq \ver{\sigma}$ a set of cardinality
$w+1$.
A \emph{\textup{(}rational\textup{)} simplicial complex} in $\R^d$ is a
finite collection $\Sigma$ of \textup{(}rational\textup{)} simplices in $\R^d$
such that any two simplices in $\Sigma$ intersect
in a common face. (This includes the case that the two simplices are disjoint:
then, and only then, they intersect in
$\emptyset$, their unique common  $(-1)$-dimensional face.) The \emph{dimension}
of $\Sigma$ is the maximum of the dimensions of its simplices. The simplices of $\Sigma$ having dimension $0$ are its
\emph{vertices}.
The \emph{support}, or \emph{underlying polyhedron}, of $\Sigma$ is
$|\Sigma|=\bigcup_{\sigma\in\Sigma}\sigma$. It indeed is a (rational) polyhedron,
by definition. Conversely, it is a basic fact that every (rational) polyhedron
$P$ is the support of some (rational) simplicial complex $\Sigma$, called a \emph{triangulation} of $P$;
see e.g., \cite[2.11]{rs}. Given any polyhedron $P\subseteq \R^{d}$, its \emph{relative interior} $\rel{P}$ is its topological interior
in the unique dimension-minimal affine subspace of $\R^{d}$ that contains $P$, endowed with the subspace topology inherited from $\R^{d}$. When
$P$ is a simplex $\sigma$ with vertices $v_{0},\ldots, v_{u}$, then $\rel{\sigma}$ consists of those convex combinations $\sum_{i=0}^{u}\lambda_{i}v_{i}$ such that $\lambda_{i}>0$ for each $i=0,\ldots,u$.

\begin{definition}\label{d:kleenet}Fix an integer $n\geq 1$.
Let
\[
   \mathcal{S}_n=\{\conv{C}
\mid
C\subseteq \{  0,\m,1\}^n \text{ and $C$ is a chain of }(\{0,\m,1\}^n,\preceq_n)\}
.
\]
An elementary computation shows that $\mathcal{S}_n$ is a  rational triangulation of $[0,1]^{n}$, called the \emph{Kleene triangulation} of the $n$-dimensional cube.
\end{definition}

A picture of the Kleene triangulation of $[0,1]^{2}$ is shown in Figure \ref{fig:standard}. We will shortly see that the Kleene triangulation enjoys a stronger property than mere rationality, known as regularity. This, however, is not needed for our next two results.
\begin{lemma}\label{lemma:tot_order_simplex_Sn}
   Fix an integer $n\geq 1$. For each simplex $\sigma\in\mathcal{S}_n$, there is a permutation
   $p\colon \{1,\ldots, n\}\rightarrow\{1,\ldots, n\}$ and functions
   $e\colon\{0,\ldots , n\}\rightarrow\{\geq,=\}$ and
   $s\colon\{1,\ldots, n\}\rightarrow\{-,+\}$ such that $\sigma$ is
   the solution set in $[0,1]^{n}$ of the system of equations and inequalities
   \begin{equation}\label{eq:eqineqsystem}
   \begin{cases}
    \m\, e(0)\, x_{p(1)}^{s(1)}\\
    x_{p(i)}^{s(i)}\, e(i)\, x_{p(i+1)}^{s(i+1)}&
\mbox{for each }i\in\{1,\ldots,n-1\}\\
    x_{p(n)}^{s(n)}\, e(n)\, 1
   \end{cases}
   \end{equation}
   where $(x_1,\ldots,x_n)\in [0,1]^{n}$ and
   $x_i^{+}=x_i$ and $x_i^{-}=1-x_i$.
   In particular, if
  ${f\in\McN_{n}}$ lies in the Kleene algebra generated by
   $\pi_1,\ldots,\pi_n, 0, 1$, where $\pi_i\colon[0,1]^{n}\rightarrow [0,1]$ is the projection map onto the $i^{\rm th}$ coordinate,
   then either~$f$ is constantly zero or constantly one over
   $\sigma$, or there is an index $1\leq i\leq n$ such that $f$ coincides
   over $\sigma$ either with $\pi_i$ or with $1-\pi_i$.
\end{lemma}
\begin{proof}
  Let $v_u \prec_{n} \cdots \prec_{n} v_1$ be the vertices of $\sigma$,
  where the order is the strict order derived from $\preceq_{n}$.
  Let $t\colon\{1, \ldots , n\}\rightarrow\{-,+\}$ be defined by
  \[
    t(i)=\begin{cases}
      - &\mbox{if }\pi_i(v_1)=0;\\
      + &\mbox{otherwise.}
    \end{cases}
  \]
  For each $i = 1,\ldots, n$, consider
  \[
  \textstyle  c_i=(\pi_i^{t(i)}(v_1),\ldots,\pi_i^{t(i)}(v_u))\in\{0,\m, 1\}^u
.
  \]
  Any chain  of $(\{0,\m,1\}^n,\preceq_n)$  is easily seen to have at most $n+1$ elements. Further, one can check that a chain of length $n+1$ must have
  the element $(\frac{1}{2},\ldots,\frac{1}{2})$ as minimum.
  Since $\{v_1,\ldots, v_u\}$ is a chain of $(\{0,\m,1\}^n,\preceq_n)$, it follows that there exists $j\in\{1,\ldots, n+1\}$ such that,
  for each $k \in \{1,\ldots,u\}$, $\pi_i^{t(i)}(v_k)=\m$ for $1 \leq k\leq j$, and $\pi_i^{t(i)}(v_k)=1$ for $j < k \leq u$. Then $c_i\in\{\m,1\}^{u}$, and there
  is a permutation $p\colon \{1,\ldots, n\}\rightarrow \{1,\ldots, n\}$ such that
  \[
     (1,\ldots,1)\geq c_{p(1)} \geq c_{p(2)} \geq \cdots  \geq c_{p(n)} \geq \textstyle (\m,\ldots,\m),
   \]
where  $\geq$ is the coordinatewise order of $\{\m,1\}^u$. Those inequalities that actually
hold with equality determine a map $e\colon \{0,\ldots, n\}\rightarrow\{\geq, =\}$. Letting $s=t\circ p$,
it  follows easily that $\sigma$ is the solution set of (\ref{eq:eqineqsystem})  in $[0,1]^{n}$. The rest is clear.
\end{proof}

\begin{lemma}\label{Cor:SimplexSolution}Fix an integer $n\geq 1$, a subset
  $\Theta\subseteq E(\widetilde{K}^{n})^{2}$, and a simplex $\sigma\in\mathcal{S}_n$. With $u_{n}^{2}$ as  in \textup{(\ref{eq:un2})}, we have
 \[
  \ver{\sigma}\subseteq \Sol_{\KK}{(\Theta)}\ 
\mbox{ iff }\ 
\sigma\subseteq \Sol_{\MM}{(u_n^{2}(\Theta))}\ 
\mbox{ iff }\ 
\rel{\sigma}\cap \Sol_{\MM}{(u_n^{2}(\Theta))}\neq\emptyset
.
 \]
Therefore, the set
\[
\Sigma_{\Theta}=\{\,\sigma\in \mathcal{S}_n \mid
 \ver{\sigma}\subseteq \Sol_{\KK}{(\Theta)}\}
\]
is a rational simplicial complex in $[0,1]^{n}$ that triangulates $\Sol_{\MM}{(u_n^{2}(\Theta))}$, i.e., satisfies
\[
     \Sol_{\MM}{(u_n^{2}(\Theta))}=|\Sigma_{\Theta}|
.
\]
\end{lemma}
\begin{proof}
 By Lemma \ref{lemma:tot_order_simplex_Sn}, each $(f,g)\in u_n^{2}(\Theta)$ is such that $f$ and $g$ are linear over $\sigma$, for each $\sigma \in \mathcal{S}_{n}$.
Hence
  $\sigma\subseteq \Sol_{\MM}{(u_n^{2}(\Theta))}$ iff ${\rel{\sigma}\cap \Sol_{\MM}{(u_n^{2}(\Theta))}\neq\emptyset}$.
From  Lemma  \ref{Th:KleeneMVSolutions}, it follows that the set $\Sol_{\KK}{(\Theta)}$ is equal to $\Sol_{\MM}{(u_n^{2}(\Theta))}\cap\{0,\m,1\}^{n}$. Thus  $\ver{\sigma}\subseteq \Sol_{\KK}{(\Theta)}$ iff ${\sigma\subseteq \Sol_{\MM}{(u_n^{2}(\Theta))}}$. The rest  is a straightforward consequence of the first statement.
\end{proof}
For the proof of Theorem I we  will need one more lemma that provides us with a combinatorial  isomorphism criterion for MV-algebras of the form $\McN{(P)}$, where $P$ is a rational polyhedron. To state the criterion, in turn, we need to introduce the notion of regular triangulation. Regularity plays a central r\^ole in the theory of MV-algebras: for further background, see~\cite{mundicibis}.

If $v \in \Q^d\subseteq\R^d$, there is a unique way to write out $v$ in
coordinates as
\[
\textstyle v = (\frac{p_1}{q_1},\ldots,\frac{p_d}{q_d})
\]
 with $p_i,q_i \in\Z$ , $q_i > 0$, $p_i$ and $q_i$ relatively prime for each $i=1,\ldots, d$.
Setting $q = {\rm lcm}{\{q_1,q_2,\ldots,q_d\}}$, the \emph{homogeneous
correspondent} of $v$ is defined to be the integer vector
\[\textstyle
 \tilde{v} = (\frac{qp_1}{q_1},\ldots,\frac{qp_d}{q_d}, q) \in \Z^{d+1} 
.
\]
The positive integer $q$ is then called the
\emph{denominator} of $v$, written
\begin{align}\label{eq:den}
\den{v}
.
\end{align}
Clearly, $\den{v}=1$ iff $v$ has integers coordinates.
A rational $u$-dimensional simplex with vertices
$v_0,\ldots,v_u$ is \emph{unimodular}, or \emph{regular}, if the set
$\{\tilde{v}_0,\ldots,\tilde{v}_u\}$ can be completed
to a $\Z$-module basis of $\Z^{d+1}$; equivalently, if there is a $(d+1)\times
(d+1)$ matrix with integer entries whose first $u$ columns are
$\tilde{v}_0,\ldots,\tilde{v}_u$, and whose determinant is $\pm 1$. A simplicial
complex is \emph{unimodular}, or \emph{regular}, if each one of its simplices is regular.
Let us remark that the Kleene triangulation~$\mathcal{S}_{n}$ of Definition \ref{d:kleenet} is regular, by a straightforward computation.

We have shown above, cf.~
(\ref{eq:nerve}), how to  construct a rational polyhedron out of a weighted abstract simplicial complex. Now we recall how to obtain an instance of the latter from a regular triangulation of the former. Let $P\subseteq \R^{d}$ be a rational polyhedron, and let $\Sigma$ be a regular triangulation of $P$. Set
\begin{align*}
\SC{(\Sigma)}=\{\ver{\sigma} \mid \sigma \in \Sigma\}
.
\end{align*}
Then, clearly, $\SC{(\Sigma)}$ is an abstract simplicial complex with set of vertices $V=\{v \in \ver{\sigma}
\mid 
  \sigma \in\Sigma\}$. Observe that there is a map $\den_{\Sigma}\colon V \to \{1,2,\ldots\}$ defined by
\begin{align*}
\den_{\Sigma}{(v)}=\den{v}
,
\end{align*}
where $\den{v}$ is the denominator of $v$ defined as in the preceding paragraph. Thus, to any regular triangulation $\Sigma$ we associate the weighted abstract simplicial complex
\begin{align}\label{eq:complbasis}
(\SC{(\Sigma)}, \den_{\Sigma})
.
\end{align}
This construction of $(\SC{(\Sigma)}, \den_{\Sigma})$ out of $P$ only requires that $\Sigma$ be a rational, not necessarily regular, triangulation of $P$. However, the following statement becomes false upon weakening regularity  to rationality.
\begin{lemma}\label{l:isocriterion}Fix integers $d,d'\geq 1$, and let $P\subseteq \R^{d}$ and $Q\subseteq \R^{d'}$ be rational polyhedra. Further, let $\Sigma$ and $\Delta$ be regular triangulations of $P$ and $Q$, respectively. If there is an isomorphism of weighted abstract simplicial complexes
\[
(\SC{(\Sigma)},\den_{\Sigma})
\cong
(\SC{(\Delta)},\den_{\Delta})
,
\]
then    the MV-algebras $\McN{(P)}$ and $\McN{(Q)}$ are isomorphic.
\end{lemma}
\begin{proof}This was first proved in \cite[Proposition 4.4. and Theorem 6.5]{mm}. Cf.~also \cite[Lemma 3.14 and Theorem 6.8]{mundicibis}.
\end{proof}

\begin{proof}
[End of Proof of Theorem I] Since $B$ is finite, there is an integer $n\geq 1$ and a congruence $\Theta$ on the free Kleene algebra $\F^{\KK}_{n}$ on $n$ generators such that $B\cong \F^{\KK}_{n}/\Theta$. Let us identify $\F^{\KK}_{n}$ with $E(\widetilde{K}^{n})$ by Lemma \ref{l:freeK}. By the natural duality for Kleene algebras, the quotient map $q \colon E(\widetilde{K}_{n})\twoheadrightarrow B$
dualises to a monomorphism $\iota =D(q) \colon D(B)\hookrightarrow \widetilde{K}_{n}$ of Kleene spaces. Direct inspection shows that, set-theoretically, $\iota$ is an injection. Hence, using Lemma \ref{Th:DualAsSolutionSetKleene}, we shall safely identify $D(B)$ with the subset of $\{0,\m,1\}^{n}$ given by
\[\textstyle
\Sol_{\KK}{(\Theta)}=\{v \in \{0,\m,1\}^{n}\mid f(v)=g(v) \text{ for each } (f,g)\in\Theta\}
,
\]
 endowed with the structure of Kleene space by restriction. It then follows by a simple verification that $B$ is isomorphic to the Kleene algebra of restrictions of elements of $E(\widetilde{K}^{n})$ to $\Sol_{\KK}{(\Theta)}$. Now, as in (\ref{e:freeover}), we have $F(B)=\F_{n}^{\MM}/\widehat{u_{n}^{2}(\Theta)}$, where $u_{n}^{2}$ is as in (\ref{eq:un2}). By the Chang-McNaughton Theorem, let us identify~$\F_{n}^{\MM}$ with $\McN_{n}$. By the  Hay-W\'ojcicki Theorem we have
\begin{align}\label{e:1}
F(B)=\McN_{n}/\widehat{u_{n}^{2}(\Theta)}\cong \McN{(\Sol_{\MM}{(u_{n}^{2}(\Theta))})}
.
\end{align}
Let $\Sigma_{\Theta}$ be the regular triangulation, defined as in Lemma \ref{Cor:SimplexSolution},
such that
\begin{align}\label{e:2}
\Sol_{\MM}{(u_{n}^{2}(\Theta))}=|\Sigma_{\Theta}|
.
\end{align}
On the other hand, consider    the the weighted abstract simplicial complex $(\N{(\Sol_{\KK}{(\Theta)})}, \omega)$ associated to $D(B)$, as in Definition \ref{d:asspoly}, and let $P_{\N{(\Sol_{\KK}{(\Theta)})}}^{\omega}$  be the companion polyhedron of $D(B)$.
From Lemma \ref{Cor:SimplexSolution} it follows  that   $\Sigma_{\Theta}=\{C\in \mathcal{S}_n \mid
C\subseteq \Sol_{\KK}{(\Theta)}\}$. Direct inspection  of the definitions then shows that
\begin{align}\label{e:3}
(\SC{(\Sigma_{\Theta})},\den_{\Sigma_{\Theta}})=(\N{(\Sol_{\KK}{(\Theta)})}, \omega)
,
\end{align}
that is,
the weighted abstract simplicial complex corresponding to the regular triangulation $\Sigma_{\Theta}$ coincides with the weighted abstract simplicial complex associated to $D(B)$. (We explicitly remark that, having identified the underlying set of $D(B)$ with $\Sol_{\KK}{(\Theta)}$ at the outset of this proof, in (\ref{e:3}) we actually have an equality, and not just an isomorphism of  weighted abstract complexes.) Using Lemma \ref{l:isocriterion} together with (\ref{e:2}) and the definition of  $P_{\N{(\Sol_{\KK}{(\Theta)})}}^{\omega}$, from~(\ref{e:3}) we  infer
\begin{align*}
\McN{(\Sol_{\MM}{(u_{n}^{2}(\Theta))}})\cong\McN{(P_{\N{(\Sol_{\KK}{(\Theta)})}}^{\omega})}
,
\end{align*}
which together with (\ref{e:1}) yields
\begin{align*}
F(B)\cong \McN{(P_{\N{(\Sol_{\KK}{(\Theta)})}}^{\omega})}
,
\end{align*}
as was to be shown.
\end{proof}

\section{Bases of MV-algebras.}\label{s:MV-duality}
In this section we discuss the  notion of basis of an MV-algebra that will eventually lead to a solution of the recognition problem.

Given a finite set $\mathcal{B}=\{b_1,\ldots,b_t\}$ of non-zero elements of an MV-algebra~$A$, we say $\mathcal{B}$ is
\emph{starrable at $\{b_r,b_s\}\subseteq \mathcal{B}$}, where $r < s$, if $b_r \wedge b_s \neq 0$.
In that case, the \emph{stellar subdivision of $\mathcal{B}$ at  $\{b_r,b_s\}$} is the set
\[
\mathcal{B}_{b_r,b_s}=\{b'_1,\ldots,b'_t,b'_{t+1}\}\setminus\{0\} 
 ,
\]
where  $\setminus$ denotes set-theoretic difference, and
\begin{align*}
b'_r &= b_r\odot\neg b_s 
,\\ \tag{*}\label{t:stellar}
b'_s &= b_s\odot\neg b_r 
,\\
b'_i &= b_i  \ \text{ for } 1\leq i \leq t,\ r\neq i \neq s
,\\
b'_{t+1} &= b_r\wedge b_s 
.\end{align*}
We  say that $B_{b_r,b_s}$ is obtained
from $B$ via a \emph{stellar subdivision} (\emph{at} $\{b_r,b_s\}$).
Cf.~\cite[Lemma~2.4 and Prop.~5.2]{mmm}; see also \cite[Prop.~2.2 and proof of Lemma~4.3]{marraja}
for an early application of stellar subdivisions, in the  sense above, to  lattice-ordered groups. The classical, geometric notion of stellar subdivision of a simplicial complex---of which we make no direct use in this paper---is of course central to piecewise linear topology;  see \cite{alexandrov, rs} for background. Here we do not attempt to explain the relationship between the algebraic and the geometric notion of stellar subdivision, as this is irrelevant to our proofs; the interested reader is referred to \cite{mmm, m, mundicibis}.
\begin{definition}\label{d:basis}Let $A$ be an MV-algebra, and let $\mathcal{B}=\{b_1,\ldots,b_t\}\subseteq A\setminus\{0\}$ for
some integer $t \geq 0$.
\begin{enumerate}
 \item $\mathcal{B}$ is \emph{$1$-regular} if whenever $\mathcal{B}$ is starrable at $\{b_r,b_s\}\subseteq \mathcal{B}$ then, with $\mathcal{B}_{b_r,b_s}$ as in (\ref{t:stellar}),
 the  following hold.
For any $1\leq i_1 < \cdots < i_{k} \leq t$, if
\[
 (b_r\wedge b_s)\wedge b_{i_1}\wedge\cdots\wedge b_{i_{k}} > 0 
\text{ holds in } A
\]
then for every $\emptyset\neq J\subseteq \{i_1,\ldots,i_{k}\}$ with $\{r,s\}\not \subseteq J$
\[
 (b_r \wedge b_s)\wedge \bigwedge_{j\in J}b'_j > 0 
\text{ holds in } A.
\]
\item $\mathcal{B}$ is \emph{regular},
if it is $1$-regular, and $\mathcal{B}_{b_r,b_s}$ is again $1$-regular, whenever~$\mathcal{B}$ is starrable at $\{b_r,b_s\}\subseteq \mathcal{B}$.

\item
$\mathcal{B}$ is a \emph{basis} of $A$ if it generates $A$, it is regular, and there are integers (called \emph{multipliers}) $1 \leq m_{1},\ldots,m_{t}$ such that, letting  $mb$ be short for $b \oplus \cdots \oplus b$ ($m$ times),
\begin{align}\label{eq:partitionofunity}
\neg b_{i} =(m_{i} -1)b_{i} \oplus \bigoplus_{i\neq j}m_{j}b_{j}
, 
\text{ for each } i =1,\ldots,t.
\end{align}
\end{enumerate}
\end{definition}
\begin{remarks} (I)
If $A$ is any MV-algebra, there exists a lattice-ordered Abelian group
$G$ endowed with a strong order unit $1$, unique to within a
unit-preser\-ving lattice-group homomorphism, such that $A$  is isomorphic to the MV-algebra $\Gamma(G,1)=\{g \in G \mid 0\leq g \leq 1\}$ with $\neg g = 1-g$ and $g_{1}\oplus g_{2}=(g_{1}+g_{2})\wedge 1$.
(For  background on lattice-ordered groups  see, e.g., \cite{bkw}).
This fact is part of Mundici's categorical equivalence between MV-algebras and lattice-orderd Abelian groups with a strong order unit, see \cite[Chapter~7]{cdm}.
When translated into the language of lattice-ordered groups,  the operation $x\odot\neg y$ featuring in  (\ref{t:stellar}) coincides with truncated subtraction in $G$, ${(x-y)\vee 0}$.
Further, the system of equations  (\ref{eq:partitionofunity}) boils down to the single equivalent condition $\sum_{i=1}^{t}m_{i}b_{i}=1$. In other words, (\ref{eq:partitionofunity}) asks that $\{b_{1},\ldots,b_{n}\}$ be a partition of unity \textup{(}=of the strong order unit $1$\textup{)} in the lattice-group  $(G,1)$ with multipliers $m_{1},\ldots, m_{t}$.
Thus \emph{a basis of an MV-algebra $A$ is a regular partition of unity that generates $A$.}

(II) 
In \cite[Definition 6.1]{mundicibis}, bases of MV-algebras are not defined as in our Definition \ref{d:basis}. However, a routine adaptation
of \cite[Lemmas 2.1 and 2.6]{m} to MV-algebras,
together with \cite[Corollary 6.4]{mundicibis}, show that the two definitions are equivalent.  
(Here one uses the fact that the equation $x \odot \neg (x \wedge y) = x \odot \neg y$ holds in each MV-algebra, as an easy check via Chang's completeness theorem shows. This is needed because stellar subdivisions in  lattice-ordered Abelian groups are defined using the term $x-(x\wedge y)$, cf.~
\cite{m}, while (\ref{t:stellar}) uses $x\odot \neg y$.) Hence, the meaning of `basis' in the present paper agrees with the one in \cite{mundicibis}.

(III) 
If $A$ is an MV-algebra with a basis $\mathcal{B}$, the multipliers of $\mathcal{B}$ are uniquely determined \cite[(iii$'$) on p.~70]{mundicibis}. We will write
\[
{\rm mult} \colon \mathcal{B} \to \{1,2,\ldots\}
\]
for the function that assigns to each $b\in\mathcal{B}$ its multiplier.\qed
\end{remarks}

Any basis $\mathcal{B}$ of an MV-algebra $A$ determines an abstract simplicial complex
%
\[
\mathcal{B}^{\bowtie} = \{C \subseteq \mathcal{B} 
\mid 
\bigwedge C > 0 \text{ in } A\}
.
\]
%
(Observe that $\emptyset \in \mathcal{B}^{\bowtie}$ whenever $0\neq 1$ in $A$, i.e., whenever $A$ is non-trivial. Indeed, the infimum of the empty set in the  bounded distributive lattice $A$ is its top element $1$. In the degenerate case $A=\{0=1\}$, the unique basis of $A$ is the empty set, and $\mathcal{B}^{\bowtie}$ is empty.)
This complex is naturally weighted by the multipliers, i.e.,
\begin{align}\label{eq:complbasis2}
(\mathcal{B}^{\bowtie},{\rm mult})
\end{align}
is a weighted abstract simplicial complex. Cf.~the weighted complex  (\ref{eq:complbasis}) arising from a regular triangulation.

Examples of bases arise in abundance from regular triangulations. Let $\Sigma$ be a regular triangulation with $|\Sigma| \subseteq \R^{d}$,  and let $v$ be one of its vertices. The \emph{Schauder hat at $v$} is the uniquely determined $\Z$-map $h_{v}\colon |\Sigma| \to \R$ which attains the value $\frac{1}{\den{v}}$
at $v$, vanishes at all remaining vertices of  $\Sigma$, and is linear on each simplex of  $\Sigma$. The \emph{Schauder basis } $H_{\Sigma}$  over $\Sigma$ (called Schauder set in [3, (9.2.1)]) is the set of hats $ \{h_{v} \mid v \text{ is a vertex of } \Sigma\}$.
\begin{lemma}\label{lemma:schauder-basis}
 For an integer $n\geq 1$ and $X\subseteq [0,1]^{n}$, suppose there is a regular triangulation
$\Sigma$ with $X=|\Sigma|$.
\begin{enumerate}[\rm(1)] 
\item The Schauder basis $H_\Sigma$ over $\Sigma$
 is a basis of $\McN{(X)}$.
\item For each vertex $v \in \Sigma$, the  Schauder hat $h_{v} \in H_{\Sigma}$ at $v$ satisfies ${\rm mult}
(h_{v})=\den{v}$.
\item The bijection $v \mapsto h_{v}$, as $v$ ranges over the vertices of $\Sigma$, and $h_{v}$ denotes the Schauder hat at $v$, is an isomorphism of weighted abstract simplicial complexes
\[
(H^{\bowtie}_{\Sigma},{\rm mult}) \cong (\SC{(\Sigma)}, \den_{\Sigma})
,
\]
where the left-hand side is defined as in  \textup{(\ref{eq:complbasis2})}, and
the right-hand side  as in \textup{(\ref{eq:complbasis})}.
\end{enumerate}
\end{lemma}
\begin{proof}A reformulation of \cite[Theorem 5.8 and Corollary 6.4]{mundicibis}.
\end{proof}
To solve the recognition problem we will need a lemma that strengthens Lemma \ref{l:isocriterion} from Schauder bases to bases {\it tout court}.
\begin{lemma}\label{l:isocriterion2}Let $A_{1}$ and $A_{2}$ be MV-algebras with bases $\mathcal{B}_{1}$ and $\mathcal{B}_{2}$,
respectively.  If there is an isomorphism of weighted abstract simplicial complexes
\[
(\mathcal{B}_{1}^{\bowtie},{\rm mult})
\cong
(\mathcal{B}_{2}^{\bowtie},{\rm mult})
,
\]
then $A_{1}$ and $A_{2}$ are isomorphic as MV-algebras.
\end{lemma}
\begin{proof}This was originally proved in \cite[Theorem 6.5]{mm} under the additional assumption that the algebras at hand be semisimple. The assumption is now subsumed by the fact that any MV-algebra having a basis (in the sense of Definition \ref{d:basis})
is automatically semisimple, cf.~\cite[Lemma 2.6 and the Remark following it]{m}. Cf.~also \cite[Theorem 6.8]{mundicibis}.
\end{proof}

\section{MV-algebras with a Kleene basis: solution to the recognition problem.}\label{s:recognition}

We now isolate those special bases of an MV-algebra that detect freeness over a finite Kleene algebra.
Let $\SC$ be any abstract simplicial complex with vertex set $V$. A subset  $N \subseteq V$ such that $N \not \in \SC$ is called a \emph{non-face} of $\SC$; and a non-face that  is inclusion-minimal, meaning that all of its proper subsets belong to $\SC$, is called a \emph{missing face} of
$\SC$. We further write
\[
\SC^{(k)} = \big\{S \in\SC \bigm| |S|=k\big\}
\]
for the \emph{$k$-skeleton} of $\SC$, i.e., the set of $(k-1)$-simplices of $\SC$. It is clear that $\SC^{(2)}$ is the same thing as a finite graph (with undirected edges, and neither loops nor multiple edges). We call such a graph $\SC^{(2)}$ a \emph{comparability}, or we say that \emph{there is a comparability over $\SC^{(2)}$}, if the edges of  $\SC^{(2)}$ can be transitively oriented, meaning that whenever edges $\{p,r_1\}, \{r_1,r_2\}, \ldots,
\{r_{u-1},r_u\}, \{r_u,q\}$ are oriented as $(p,r_1),(r_1,r_2),\ldots,$ $(r_{u-1},r_u),$ $(r_u,q)$, then there is an edge $\{p,q\}$ oriented as $(p,q)$. Finally, we recall that a \emph{sink} of a directed graph is a vertex $p$ of the graph such that no (directed) edge of the graph has $p$ as first element.
\begin{definition}\label{d:kleenebasis}
Let $A$ be an MV-algebra. A basis $\mathcal{B}$  of $A$ is called a \emph{Kleene basis}  if it satisfies the following conditions.
     \begin{itemize}
       \item[{\rm (a)}] The multiplier of each element of $\mathcal{B}$ is either $1$ or $2$, i.e., the range of ${\rm mult}\colon \mathcal{B}\to \{1,2,\ldots\}$ is contained in $\{1,2\}$.
       \item[{\rm (b)}] The abstract simplicial complex $\mathcal{B}^{\bowtie}$ has no missing face of cardinality three or greater.
       \item[{\rm (c)}] There is a comparability over the graph $(\mathcal{B}^{\bowtie})^{(2)}$ such that each element of $\mathcal{B}$ having multiplier $1$ is a  sink.
       \end{itemize}
\end{definition}
%
 Our second main result is:
\begin{theoremII}
Let $A$ be  any MV-algebra. Then $A$ is free over some finite Kleene algebra if, and only if, $A$ has a Kleene basis.
\end{theoremII}
%
For the proof we  prepare several lemmas. First we recall how to  recognise order complexes among abstract simplicial complexes.
\begin{lemma}\label{lemma:orderedcomplexes}
An abstract simplicial complex $\SC$ is the nerve of some \textup{(}finite\textup{)} partially ordered set if, and only if, it
has no missing faces of cardinality $\geq 3$, and the graph $\SC^{(2)}$ is a comparability.
\end{lemma}
\begin{proof}This is \cite[Proposition 4]{bayer};  further references related to this result are given in that paper.
\end{proof}
%
We next establish a strengthening of Lemma \ref{Cor:SimplexSolution}.
\begin{lemma}\label{lemma:triangofX}
Fix an integer $n\geq 1$, and a subset $X\subseteq [0, 1]^n$. The following are
equivalent.
\begin{itemize}
\item[\textup{(i)}] $X=\Sol_{\MM}{(u_n^{2}(\Theta))}$ for some finite set $\Theta\subseteq E(\widetilde{K}^{n})^{2}$, where $u_{n}^{2}$ is as in \textup{(\ref{eq:un2})}.
\item[\textup{(ii)}] There exists a subposet $(W,\preceq)$ of $(\{0,\m,1\}^n,\preceq_n)$ such that $X$ is triangulated by the nerve of  $(W,\preceq)$, meaning that the regular simplicial complex
\[
\Sigma_{W}=\{\conv{C} 
\mid
 C \in \N{(W)}\}
\]
satisfies
\begin{align*}\label{eq:triang}
X = |\Sigma_{W}|
.
\end{align*}
\end{itemize}
\end{lemma}
\begin{proof}
(i) $\Rightarrow$ (ii). This  follows easily from Lemma \ref{Cor:SimplexSolution}, together with a trivial computation to verify that $\Sigma_{W}$ is regular.

(ii) $\Rightarrow$ (i).  Given $(W,\preceq)$ as in the hypothesis, consider the Kleene space
\[
(W,\preceq,\sim,M)
\]
obtained from $\widetilde{K}^{n}=(\{0,\m,1\}^{n}, \preceq_{n}, \sim_{n}, \{0,1\}^{n})$  by restricting  all relations to $W$ as in Lemma \ref{Th:DualAsSolutionSetKleene}. Then $W=\Sol_{\KK}{(\Ker{E(\iota)})}$,
where ${\iota\colon (W,\preceq,\sim,M)\hookrightarrow \widetilde{K}^{n}}$ is the  monomorphism of Kleene spaces induced by the inclusion $W\subseteq \{0,\m,1\}^{n}$. Set
\[
 \Theta=\Ker{E(\iota)}
.
 \]
Then  by Lemmas \ref{Th:DualAsSolutionSetKleene} and \ref{Th:KleeneMVSolutions} we have
\begin{align}\label{Eq:Solution_vertices}
\textstyle W=\Sol_{\KK}{(\Theta)}=\Sol_{\MM}{(u_n^{2}(\Theta))}\cap\{0,\m,1\}^{n}.
\end{align}
To complete
the proof we show that $X=\Sol_{\MM}{(u_n^{2}(\Theta))}$.

Suppose first $x\in X$. Since $X$ is triangulated by $\N(W)$, there
is a unique abstract simplex $C=\{v_1,\ldots,v_k\} \in\N(W)$ such that
${x\in \rel{\conv{C}}}$. From the definition of the Kleene triangulation $\mathcal{S}_{n}$ (Definition \ref{d:kleenet}), together with the fact that the order $\preceq$ on $W$ is the restriction of the order $\preceq_{n}$ on $\{0,\m,1\}^{n}$, we see that $\sigma=\conv{C}\in\mathcal{S}_{n}$.
Now $C=\ver{\sigma}\subseteq \Sol_{\KK}{(\Theta)}$ by (\ref{Eq:Solution_vertices}), because $C\subseteq W$. Using Lemma \ref{Cor:SimplexSolution} we therefore deduce $\sigma\subseteq \Sol_{\MM}{(u_n^{2}(\Theta))}$. In particular, since $x \in \sigma$ by construction, we have
$x\in \Sol_{\MM}{(u_n^{2}(\Theta))}$.

Conversely, suppose $x\in\Sol_{\MM}{(u_n^{2}(\Theta))}$, and let $\sigma$  be the unique
simplex of~$\mathcal{S}_n$ such that $x\in\rel{\sigma}$. Then $\rel{\sigma}\cap \Sol_{\MM}{(u_n^{2}(\Theta))}\neq\emptyset$, and again by Lemma \ref{Cor:SimplexSolution} we infer $\ver{\sigma}\subseteq \Sol_{\KK}{(\Theta)}$. By  (\ref{Eq:Solution_vertices}) we obtain  $\ver{\sigma}\subseteq W$. By the definition of the Kleene triangulation $\mathcal{S}_{n}$ (Definition \ref{d:kleenet}), $\ver{\sigma}$ is a chain in $(\{0,\m,1\}^{n},\preceq_{n})$,
hence a chain in $(W,\preceq)$, and therefore
\begin{align}\label{eq:innerve}
\ver{\sigma}\in\N{(W)}
.
\end{align}
 Since $X=|\Sigma_{W}|$ by hypothesis, (\ref{eq:innerve}) entails  $\sigma\subseteq X$, whence $x\in X$.
This completes the proof.
\end{proof}
We now characterise those weighted abstract simplicial complexes that arise as weighted nerves of Kleene spaces.
\begin{definition}\label{d:kleenec}A weighted abstract simplicial complex $(\SC,\omega)$ over the (finite) set $V$ is a \emph{Kleene complex} if it satisfies the following conditions.
\begin{itemize}
         \item[(a)] $\omega(v)\in\{1,2\}$ for each $v\in V$.
         \item[(b)] The abstract simplicial complex $\SC$ has no missing face of cardinality three or greater.
         \item[(c)] There is a comparability over the graph $\SC^{(2)}$ such that if $\omega(v)=1$, then~$v$ is a sink.
      \end{itemize}
\end{definition}
\begin{remark}\label{r:compare}Comparing Definitions \ref{d:kleenec} and \ref{d:kleenebasis}, we see that \emph{a basis ${\mathcal{B}}$ is a Kleene basis if, and only if, $(\mathcal{B}^{\bowtie},{\rm mult})$ is a Kleene complex}, where $(\mathcal{B}^{\bowtie},{\rm mult})$ is as in (\ref{eq:complbasis2}).
\end{remark}

\begin{lemma}\label{Lem:ontohom}
   A weighted abstract simplicial complex $(\SC,\omega)$ is a Kleene complex if, and only if,   there exist an integer $n\geq 1$ and a subposet $(W,\preceq)$ of $(\{0,\m,1\}^n,\preceq_n)$ such that $(\SC,\omega)$ is isomorphic to the
   weighted abstract simplicial complex $(\N{(W)},\den)$.
\end{lemma}
\begin{proof}
($\Leftarrow$). 
Let $V$ be the underlying set of $\SC$, and let $f\colon W\to V$ be an isomorphism of the weighted abstract simplicial complexes
$(\N{(W)},\den)$  and $(\SC,\omega)$.
Since $W\subseteq \{0,\m,1\}^{n}$ we have $\den{w}\in\{1,2\}$ for each $w\in W$, and therefore $\omega(v)\in\{1,2\}$ for each $v\in V$, because $\omega\circ f=\den$.
From Lemma \ref{lemma:orderedcomplexes} it follows that $\N{(W)}$ has no missing faces of cardinality $\geq 3$, and so the same holds for $(\SC,\omega)$. Since $\N{(W)}$ is the nerve of the poset $(W,\preceq)$, it follows at once that  $\preceq$ induces a comparability over $\N^{(2)}{(W)}$, and therefore $\SC^{(2)}$ is a comparability, too. If $v\in V$ is such that $\omega(v)=1$, then there exists $w\in W$ such that  $f(w)=v $ and $\den{w}=1$. By the definition of $\preceq$, it follows that $w\in\max{W}$, that is, $w$ is a sink in the comparability induced by~$\preceq$. Therefore,~$v$ is a sink in the corresponding comparability on $\SC^{(2)}$ induced by~$f$. 

($\Rightarrow$). 
 By Lemma \ref{lemma:orderedcomplexes} and (b--c) there exists a partial order $\leq$ on the underlying set  $V$ of $\SC$ such that  the abstract simplicial complex $\SC$ is isomorphic to the nerve $\N{(V)}$ of $(V,\leq)$. Since each $v\in V$ such that $\omega(v)=1$ is a sink in the comparability induced by $\leq$, we have $\omega^{-1}(1)\subseteq \max{V}$.

\begin{claimnonum}  Say $V=\{v_1,\ldots, v_m\}$ for some integer $m\geq 0$. Then there exists an order-embedding
\[\textstyle
\varphi\colon (V,\leq)\hookrightarrow (\{0,\m,1\}^{m+1},\preceq_{m+1})
\]
such that $\varphi^{-1}(\{0,1\}^{m+1})=\omega^{-1}(1)$.
\end{claimnonum}

\begin{proof}[Proof of Claim.] We define $\varphi: V\rightarrow \{0,\m,1\}^{m+1}$ by
\[
      \varphi(v_i)=(d(i),\delta_1^i,\ldots,\delta_m^i)
,
\]
   where
\[
     d(i)=\begin{cases}
     1&\mbox{ if }\omega(v_i)=1;\\
     \m&\mbox{otherwise,}
     \end{cases}
\]
   and
\[
      \delta_j^i=\begin{cases}
      1 & \mbox{if  }v_j\leq v_i;\\
      0 & \mbox{if  }v_i \in \max{V}\mbox{ and }v_j \not\leq v_i;\\
      \m & \mbox{otherwise
;}
      \end{cases}
\]
for each $i,j\in\{1,\ldots,m\}$.
It is clear that $\varphi$ is an injection.

To prove that $\varphi$ is  order-preserving, suppose   $v_i\leq v_j$. Since $\m\preceq 0,1$, to show that $\varphi(v_{i})\preceq_{m+1}\varphi(v_{j})$ it suffices to prove the following
for each $k \in \{1,\ldots,m\}$.
   \begin{enumerate}
    \item $d(i)=1$ implies $d(j)=1$.
    \item $\delta_{k}^{i}=1$ implies $\delta_k^j=1$.
    \item $\delta_{k}^{i}=0$ implies $\delta_k^j=0$.
   \end{enumerate}
%
 Proof of 1. If $d(i)=1$ then $v_i\in \omega^{-1}(1)\subseteq \max{V}$. Therefore $v_i=v_j$, which implies $d(j)=d(i)=1$.

\noindent 
Proof of 2. If $\delta_k^{i}=1$, then $v_k\leq v_i\leq v_j$. Thus $\delta_k^{j}=1$.

\noindent 
Proof of 3. If $\delta_k^{i}=0$, then $v_i\in\max{V}$. Thus $v_i=v_j$, which implies
 ${\delta_k^{i}=\delta_k^{j}=0}$.

To show that $\varphi$ is an order-embedding,  suppose $\varphi(v_i)\preceq_{m+1}\varphi(v_j)$ for some $v_{i},v_{j}\in V$. Then $\delta_i^{i}=1=\delta_{i}^{j}$, which implies $v_i\leq v_j$.

It remains to show that  $\varphi^{-1}(\{0,1\}^{m+1})=\omega^{-1}(1)$.   By direct inspection of the definition of $\varphi$ and $d(i)$,  $\varphi^{-1}(\{0,1\}^{m+1})\subseteq \omega^{-1}(1)$. For the other inclusion, observe that if $\omega(v_i)=1$ then $d(i)=1$ and, since $\omega^{-1}(1)\subseteq\max{V}$, $\delta_{j}^{i}\in\{0,1\}$ for each $j \in \{1,\ldots,m\}$. Therefore $\varphi(v_i)\in\{0,1\}^{m+1}$.
 \end{proof}

Now let $(W,\preceq)$ denote  the subposet of $(\{0,\m,1\}^{m+1},\preceq_{m+1})$ whose underlying set is $W=\varphi(V)$. By the Claim, $\varphi$ determines an isomorphism of abstract simplicial complexes between $\SC\cong\N{(V)}$ and $\N({W})$.
Using the fact that $\varphi^{-1}(\{0,1\}^{n})=\omega^{-1}(1)$, we have $\omega(v)=\den(\varphi(v))$ for each $v \in V$.  Hence
$\varphi$ determines an isomorphism of weighted abstract simplicial complexes, which proves  ${(\SC,\omega)\cong(\N{(V)},\omega)\cong(\N{(W)},\den)}$.
\end{proof}

\begin{proof}
[Proof of Theorem II]
$(\Rightarrow)$. 
Suppose $A$ is free over the finite Kleene algebra~$B$, i.e., $A\cong F(B)$. Let $D(B)=
(W,\leq,R,M)$ be the Kleene space dual to~$B$,  let $(\N{(W)}, \omega)$ be its associated weighted abstract simplicial complex, and let $P_{\N{(W)}}^{\omega}$ be its companion polyhedron. Then, by Theorem I,
\[
A \cong F(B) \cong \McN{(P_{\N{(W)}}^{\omega})}
.
\]
It suffices to prove that $\McN{(P_{\N{(W)}}^{\omega})}$ has a Kleene basis, since Kleene bases are clearly preserved by isomorphisms of MV-algebras. By  its very definition~(\ref{eq:nerve}),
the polyhedron $P_{\N{(W)}}^{\omega}$ comes with a regular triangulation $\Sigma$ that satisfies
\begin{align}\label{eq:iso}
(\SC{(\Sigma}),\den_{\Sigma})\cong(\N{(W)}, \omega)
,
\end{align}
the indicated isomorphism being of weighted abstract simplicial complexes, where $(\SC{(\Sigma}),\den_{\Sigma})$ is defined as in (\ref{eq:complbasis}). By Lemma \ref{lemma:schauder-basis} the Schauder basis~$H_{\Sigma}$ over $\Sigma$ is a basis of $\McN{(P_{\N{(W)}}^{\omega})}$, and the weighted abstract simplicial complexes $(\SC{(\Sigma)},\den_{\Sigma})$ and $(H^{\bowtie}_{\Sigma},{\rm mult})$ are isomorphic. By (\ref{eq:iso}) we obtain
\begin{align*}
(H^{\bowtie}_{\Sigma},{\rm mult})\cong(\N{(W)}, \omega)
,
\end{align*}
which by Lemma \ref{Lem:ontohom}  entails at once that $H_{\Sigma}$ is a Kleene basis of $\McN{(P_{\N{(W)}}^{\omega})}$.

 $(\Leftarrow)$. 
Let $\mathcal{B}$ be a Kleene basis of $A$. In light of Remark \ref{r:compare}, by Lemma \ref{Lem:ontohom} there is an integer  $n\geq 1$ and a subposet $(W,\preceq)$ of $(\{0,\m,1\}^{n},\preceq_n)$ such that there exists an isomorphism of weighted abstract simplicial complexes
 \begin{align}\label{eq:firstofall}
 (\mathcal{B}^{\bowtie},{\rm mult})
\cong 
(\N{(W)},\omega)
.
 \end{align}
 Set
 \begin{align}\label{eq:satisfies}
 X = \bigcup\{\conv{C}
\mid
C\in \N{(W)}\}
.
 \end{align}
Then  $X$ is triangulated by the regular simplicial complex $\Sigma$ whose simplices are the sets $\conv{C}$, as $C$ ranges in $\N{(W)}$. By  Lemma  \ref{lemma:schauder-basis}, the Schauder basis $H_{\Sigma}$ over $\Sigma$ is a basis of  $\McN{(X)}$, and there is an   isomorphism of weighted abstract simplicial complexes
\begin{align}\label{eq:moreover}
(\N{(W)},\omega) 
\cong 
(H^{\bowtie}_{\Sigma},{\rm mult})
.
\end{align}
By (\ref{eq:firstofall}) and (\ref{eq:moreover}) we have  $(\mathcal{B}^{\bowtie},{\rm mult})\cong (H^{\bowtie}_{\Sigma},{\rm mult})$. This, by Lemmas \ref{lemma:schauder-basis} and~\ref{l:isocriterion2}, implies the existence of an isomorphism of MV-algebras
\begin{align}\label{eq:congfinal}
A\cong\McN{(X)}
.
\end{align}
To complete the proof, in light of (\ref{eq:congfinal}) it suffices to show that $\McN{(X)}$ is free over a finite Kleene algebra.
Now, by its very definition (\ref{eq:satisfies}), $X$ satisfies  condition
 (ii) in  Lemma \ref{lemma:triangofX}, and therefore there exists $\Theta\subseteq E(\widetilde{K}^{n})^{2}$ such that $X=\Sol_{\MM}{(u_n^{2}(\Theta))}$. By Lemma \ref{Th:DualAsSolutionSetMV} we have
 \[
 \McN{(X)}=\McN{(\Sol_{\MM}{(u_n^{2}(\Theta)))}}\cong \McN_{n}/\widehat{u_n^{2}(\Theta)}
,
 \]
 which displays $\McN(X)$ as the MV-algebra freely generated by  the finite Kleene algebra $E(\widetilde{K}^n)/\Theta'$, where $\Theta'$ is the congruence on $E(\tilde{K}^{n})$ generated by $\Theta$.

\end{proof}

\section{Concluding remarks.}\label{s:conclusion}
%
%
One may wonder which uniqueness properties, if any, hold for MV-algebras freely generated by Kleene algebras.

\begin{question} If $F(B)$ and $F(B')$ are MV-algebras free over the
Kleene algebras $B$ and $B'$, and $F(B)$ is isomorphic to $F(B')$, must $B$ and $B'$ be isomorphic as Kleene algebras?
\end{question}
%
 The answer is readily seen to be negative.
Let $K=\{0,\m,1\}$ be the 3-element  Kleene algebra, and let $K^{2}$ be its direct product. Set
$K' = K^2 \setminus \{(0,1),(1,0)\}$, and observe that $K'$ is a subalgebra of $K^{2}$ with the operations inherited by restriction.  Obviously $K^2$ is not isomorphic to $K'$. Computation shows that
the Kleene space dual to $K^{2}$ is
\[
D(K^2)=\bigl(\,
\{a,b\},\{(a,a),(b,b)\},\{(a,a),(b,b)\},\emptyset\,
\bigr)
,
\]
while the Kleene space dual to $K'$ is
\[
D(K')=\bigl(
\{a,b\},\{(a,a),(b,b)\},\{(a,a),(b,b),(a,b),(b,a)\},\emptyset\,) 
.
\]
Thus both spaces consist of a trivially ordered doubleton, and both  have as distinguished subset of their maximal elements the empty set.
Using Definition \ref{d:asspoly} we therefore see that the weighted abstract simplicial complex associated to each space is
\[
(
\N{(D(K^{2}))},\,\omega
)=(
\N{(D(K'))},
\omega
)=\left(\,
\big\{
\{a\},\{b\}
\big\},\,\omega\,\right)
\]
where $\omega \colon \{a,b\} \to \{1,2\}$ is
the  function constantly equal to $2$. By Theorem I we have $F(K^2)\cong F(K')$ immediately, and,  with some further computation, $F(K^2)\cong F(K')\cong K^{2}$.\qed

The answer to Question 1 remains negative even if one asks that $F(B)$ be a free  MV-algebra, although producing a counterexample requires more work.
The MV-algebra freely generated by $n$ elements is certainly free over the  Kleene algebra freely generated by  $n$ elements; the question is whether it can also be free over some other Kleene algebra.

\begin{question} Can the free MV-algebra $\F_{n}^{\MM}$ on $n$ generators, for $n\geq 1 $ an integer, be free over a finite Kleene algebra that is not a free Kleene algebra?
\end{question}
%
 We show that the answer is affirmative by  exhibiting a non-free Kleene algebra~$B$ such that the free
MV-algebra over $2$ generators is free over $B$.
We describe $B$ through its dual Kleene space $D(B)$.
Recall that the dual space of the free $2$-generated Kleene algebra is $\widetilde{K}^{2}$;
 its underlying poset is depicted in Figure~\ref{fig:k2}.
\begin{figure}[h]
\centering
{\includegraphics[scale=0.4]{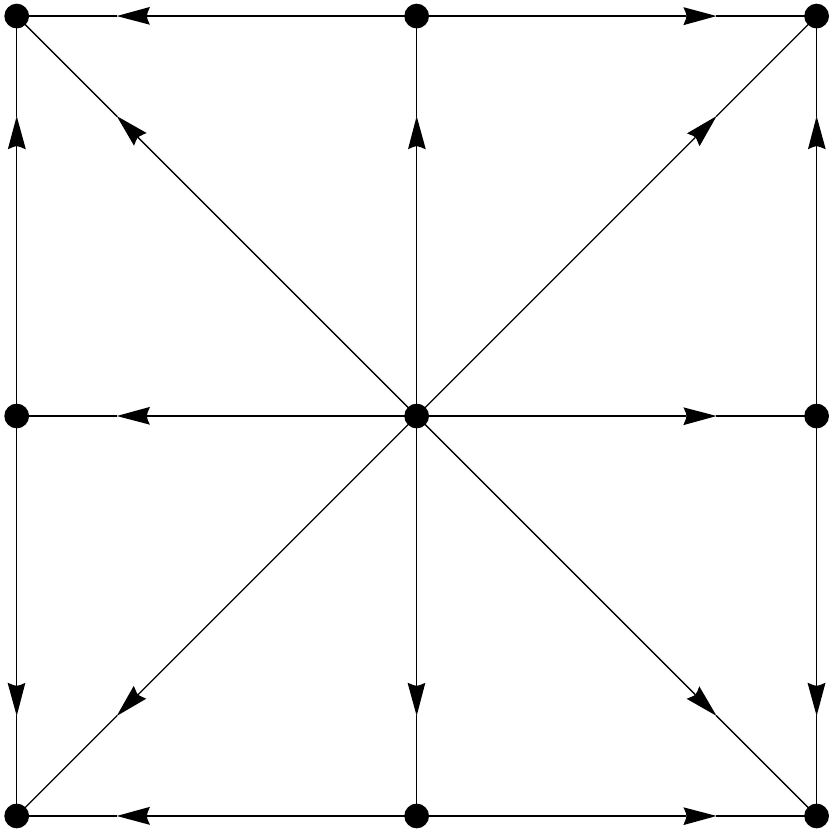}}
\caption{The Kleene (regular) triangulation $\mathcal{S}_{2}$ of $[0,1]^2$, endowed with a comparability.}
\label{fig:standard}
\end{figure}
Figure \ref{fig:standard} depicts the Kleene  triangulation $\mathcal{S}_{2}$ of $[0,1]^2$.
The underlying poset $\{0,\m,1\}^{2}$ of $\widetilde{K}^{2}$ induces on $\mathcal{S}_{2}$ the comparability described by the arrows
depicted in Figure \ref{fig:standard}. The labels in Figure \ref{fig:k2} are then the
denominators of the corresponding points of $\{0,\m,1\}^2\subseteq [0,1]^{2}$.

Now, consider the subset $W=\bigl(
(\{0,\m,1\}^{2} \times \{0\})
\setminus \{(0,\m,0)\}
\bigr) \cup \{(1,\m,\m)\}$ of $(\{0,\m,1\}^{3},\preceq_3)$ shown in Figure \ref{fig:w}.a. Then $W$ is a poset under the order inherited from $(\{0,\m,1\}^{3},\preceq_3)$; its Hasse diagram is depicted in Figure \ref{fig:w}.b. The  edges connecting elements of  $W$ in Figure \ref{fig:w}.a constitute the $2$-skeleton of an order complex isomorphic to $\N(W)$.
 Let $B$ be the
Kleene algebra whose dual space $D(B)$ is the subobject of $\widetilde{K}^{3}$ whose underlying poset
is
$W$.
The labelling in Figure \ref{fig:w}.b is induced by the distinguished subset $M$ of maximal elements of $D(B)$: elements labelled $1$ lie in $M$, elements labelled $2$ do not. A moment's reflection shows that $B$ is not a free Kleene algebra. By Theorem I, the MV-algebra $\McN{(P)}$ of $\Z$-maps on the geometric realization $P$ of the weighted nerve $(\N{(W)},\omega)$, where the range of $\omega$ is indicated by the labels in  Figure \ref{fig:w}.b, is free over $B$.
\begin{figure}[h]
\centering
\subfigure[\text{The subset $W$ of $\{0,\m,1\}^{3}$.}]
{\includegraphics[height=38mm, width=43mm]{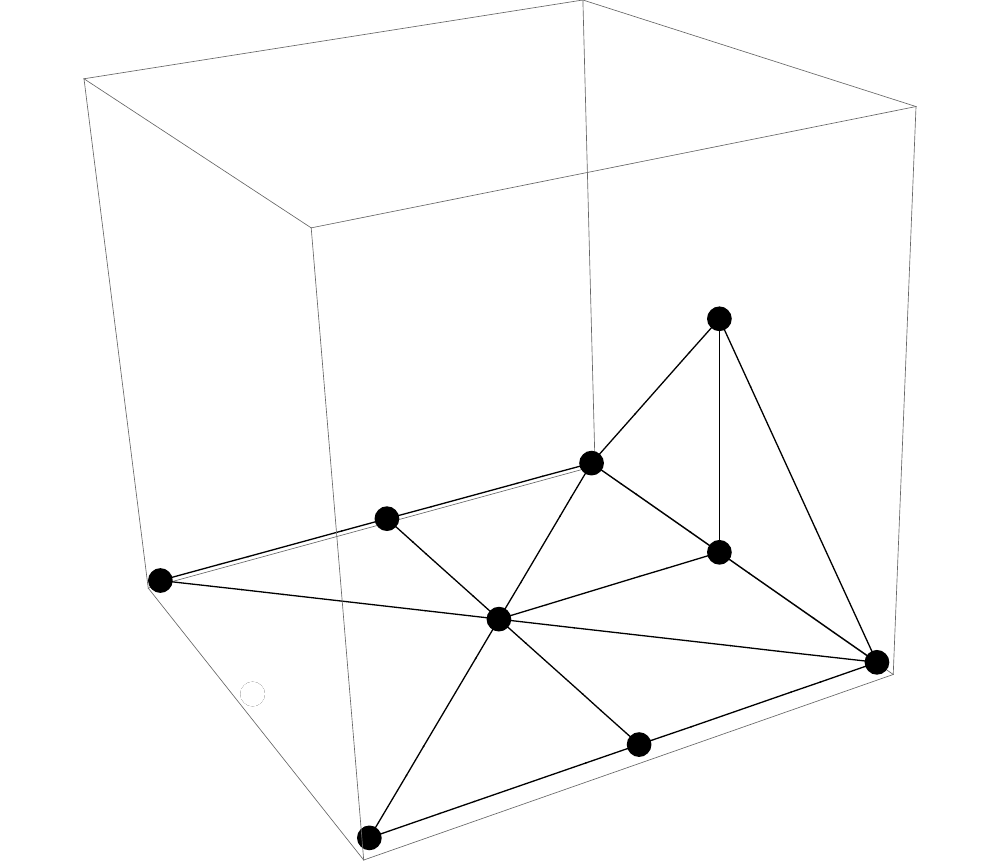}}
\qquad
\subfigure[
\text{The poset $W$.}
]
{\includegraphics[scale=0.6]{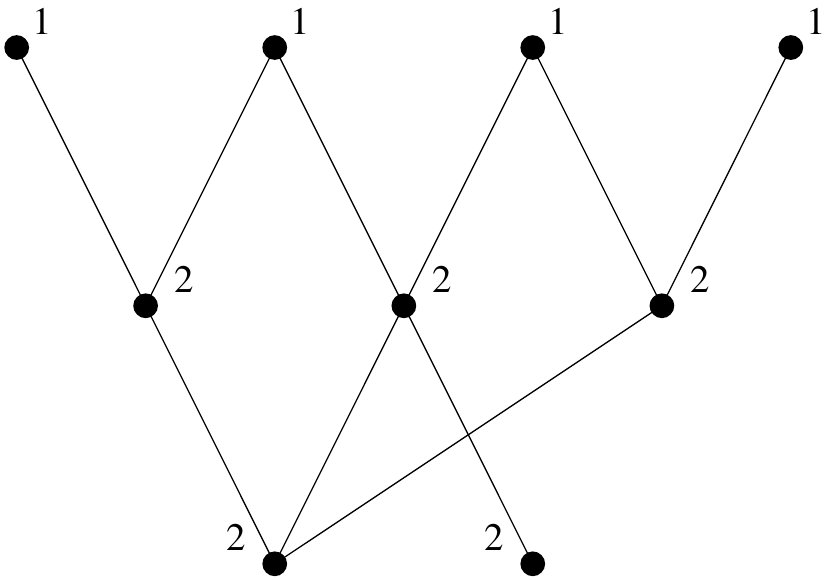}}
\caption{The underlying poset $W$ of the Kleene space dual to $B$, with labelled nodes.}
\label{fig:w}
\end{figure}
Recall that the geometric realization $P$ of $(\N{(W)},\omega)$ is defined as in~(\ref{eq:nerve}), and therefore it comes with its own regular triangulation $\Sigma$.  Then, using Lemma \ref{lemma:schauder-basis}, there is an isomorphism of weighted abstract simplicial complexes
\begin{align}\label{eq:prelast}
(
\N{(W)},
\omega
) 
\cong 
(
H^{\bowtie}_{\Sigma},
\den_{\Sigma}
)
.
\end{align}
Since $W$ has cardinality $9$, $P$ is a rational polyhedron in $[0,1]^{9}\subseteq \R^{9}$, and the Schauder basis $H_{\Sigma}$ has $9$ hats.

Next
consider the regular triangulation $\Delta$ of $[0,1]^{2}\subseteq \R^{2}$ shown in Figure \ref{fig:regular}, and its associated Schauder basis $H_{\Delta}$.  Direct inspection using $W$ shows that there is an isomorphism of weighted abstract simplicial complexes
\begin{align}\label{eq:last}
(
\N{(W)},
\omega
) 
\cong
 (
H^{\bowtie}_{\Delta},
\den_{\Delta}
)
.
\end{align}
(The comparability on the $2$-skeleton of $\Delta$  displayed in Figure \ref{fig:regular} is the one induced by the isomorphism (\ref{eq:last}).)
By Lemma \ref{l:isocriterion2}, (\ref{eq:prelast}--\ref{eq:last})  imply that the MV-algebras $\McN{(P)}$ and $\McN_{2}$ are isomorphic.  Since the former is free over $B$ by Theorem I, the latter is, too.
\begin{figure}[h]
\centering
{\includegraphics[scale=0.4]{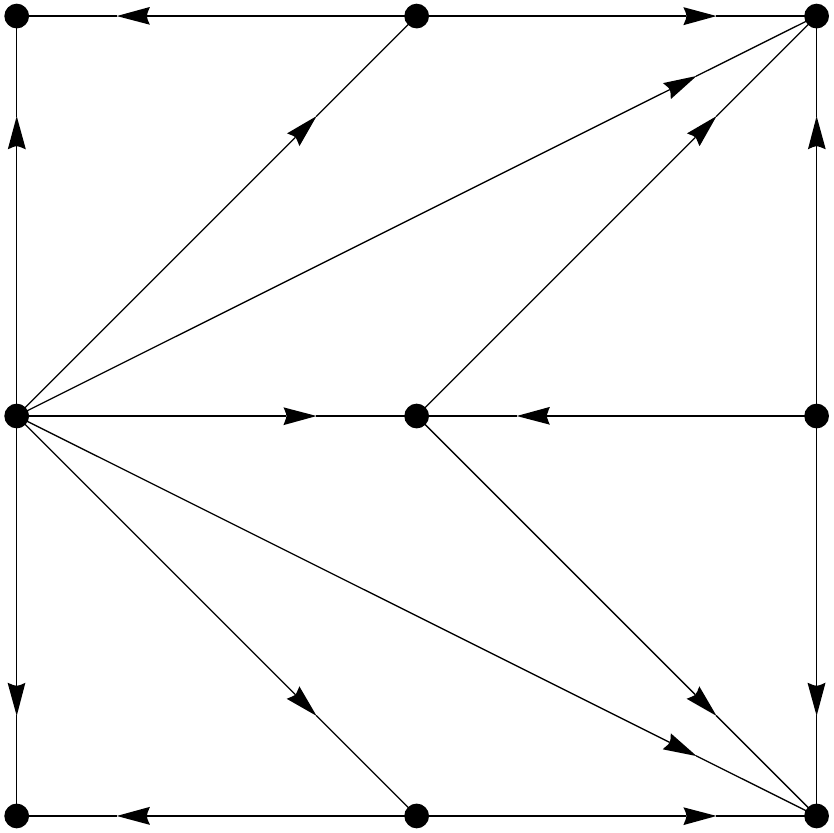}}
\caption{A regular triangulation $\Delta$ of $[0,1]^{2}$, endowed with a comparability.}
\label{fig:regular}
\end{figure}
By the Chang-McNaughton Theorem, $\McN_{2}$ is the free $2$-generated MV-algebra, which is therefore freely generated by the non-free Kleene algebra $B$. 
\bibliographystyle{plain}

\end{document}